\documentclass[letterpaper, 10 pt, conference]{ieeeconf}  % Comment this line out if you need a4paper

\IEEEoverridecommandlockouts                              % This command is only needed if 
\usepackage{amsmath,amssymb}
\usepackage{mathtools}
\usepackage{amsthm}
\usepackage{amsfonts}
\usepackage{cite}
\usepackage{comment}
\usepackage{multirow}
\usepackage{color}

\newcommand{\xVar}{\boldsymbol{x}}
\newcommand{\ID}{i}
\newcommand{\dimXVar}{n}
\newcommand{\dimID}{m}
\newcommand{\const}{d}
\newcommand{\constC}{c}
\newcommand{\ineqVar}{\lambda}
\newcommand{\ineqVars}{\boldsymbol{\ineqVar}}
\newcommand{\eqVar}{s}
\newcommand{\KKTVars}{(\ineqVars,\,\eqVar)}
\newcommand{\KKTVar}{(\ineqVar,\,\eqVar)}
\newcommand{\targetFcnHead}{J}
\newcommand{\targetFcn}{\targetFcnHead(\xVar,\,\ID)}
\newcommand{\realProbHead}{p}
\newcommand{\realProb}{\realProbHead(\ID)}
\newcommand{\probHead}{\hat{p}}
\newcommand{\prob}{\probHead(\ID)}
\newcommand{\refProbHead}{\realProbHead_0}
\newcommand{\refProb}{\refProbHead(\ID)}
\newcommand{\densRatioHead}{r}
\newcommand{\densRatio}{\densRatioHead_\ID}
\newcommand{\inner}{\expectation{\prob}{\targetFcn}}
\newcommand{\innerMax}{\max_{\probHead\in\ambiguity}\quad\inner}

\newcommand{\innerHat}{\max_{\probHead\in\hatAmbiguity}\quad\inner}

\newcommand{\innerDual}{\expectation{\refProb}{\dualExtended}}
\newcommand{\innerMin}{\inf_{\KKTVars\in\realPlus^\dimID\times\real}\quad\innerDual}

\newcommand{\lagrangianHead}{l_{\xVar}}
\newcommand{\lagrangian}{\lagrangianHead(\densRatioHead_1,\,\cdots,\,\densRatioHead_\dimID,\,\ineqVars,\,\eqVar)}
\newcommand{\lagrangeDualHead}{g_{\xVar}}
\newcommand{\lagrangeDual}{\lagrangeDualHead(\ineqVars,\,\eqVar)}
\newcommand{\dualHead}{h_\mathrm{DR}}
\newcommand{\dualExtendedHead}{\tilde{h}_\mathrm{DR}}
\newcommand{\dual}{\dualHead(\xVar,\,\ID,\,[\ineqVars]_\ID,\,\eqVar)}
\newcommand{\dualExtended}{\dualExtendedHead(\xVar,\,\ID,\,[\ineqVars]_\ID,\,\eqVar)}

\newcommand{\ambiguityKnapsack}{\mathcal{Z}}

\newcommand{\ambiguityC}{\ambiguityKnapsack_\constC}

\newcommand{\targetVaR}{\VaR{\beta}{\refProb}{\targetFcn}}
\newcommand{\targetCVaR}{\CVaR{\beta}{\refProb}{\targetFcn}}
\newcommand{\targetHatCVaR}{\hatCVaR{\beta}{\refProb}{\targetFcn}}
\newcommand{\targetTildeCVaR}{\tildeCVaR{\beta}{\refProb}{\targetFcn}}
\newcommand{\innerCVaR}{(1+\const)\,\expectation{\refProb}{\max\{0,~\targetFcn-\eqVar\}}+\eqVar}
\newcommand{\innerBetaCVaR}{(1-\beta)^{-1}\,\expectation{\refProb}{\max\{0,~\targetFcn-\eqVar\}}+\eqVar}
\newcommand{\innerCVaRFcnHead}{F_\beta}
\newcommand{\innerCVaRFcn}{\innerCVaRFcnHead(\xVar,\,\eqVar)}

\newcommand{\innerMaxSL}{\max_{\probHead\in\ambiguity}\inner}
\newcommand{\dualHeadL}{h_\mathrm{L^2}}
\newcommand{\dualExtendedHeadL}{\tilde{h}_\mathrm{L^2}}
\newcommand{\dualL}{\dualHeadL(\xVar,\,\ineqVar,\,\eqVar)}
\newcommand{\dualExtendedL}{\dualExtendedHeadL(\xVar,\,\ineqVar,\,\eqVar)}
\newcommand{\innerDualL}{\dualExtendedL}
\newcommand{\lagrangianHeadL}{l_{\xVar}}
\newcommand{\lagrangianL}{\lagrangianHeadL(\densRatioHead_1,\,\cdots,\,\densRatioHead_\dimID,\,\ineqVar,\,\eqVar)}
\newcommand{\lagrangeDualHeadL}{g_{\xVar}}
\newcommand{\lagrangeDualL}{\lagrangeDualHeadL(\ineqVar,\,\eqVar)}
\newcommand{\innerMinL}{\inf_{\KKTVar\in\realPlus\times\real}\quad\innerDualL}

\newcommand{\ambiguityTV}{\ambiguity_\mathrm{TV}}
\newcommand{\ambiguityL}{\ambiguity_\mathrm{L^2}}
\newcommand{\ambiguityDR}{\ambiguity_\mathrm{DR}}
\newcommand{\relint}{\mathrm{relint}}
\newcommand{\vect}{\mathrm{vec}}
\newcommand{\eye}{\boldsymbol{I}}
\newcommand{\diag}{\mathrm{diag}}
\newcommand{\lnFcn}[1]{\ln\left(#1\right)}
\newcommand{\expFcn}[1]{\exp\left(#1\right)}
\newcommand{\maxFcn}[1]{\max\left\{#1\right\}}

\newcommand{\anyVector}{\boldsymbol{v}}
\newcommand{\anyMatrix}{\boldsymbol{C}}
\newcommand{\anySet}{\mathcal{S}}

\newcommand{\anyFcn}{f}
\newcommand{\anyRandom}{s}
\newcommand{\anyEvent}{A}

\newcommand{\real}{\mathbb{R}}
\newcommand{\realPlus}{[0,\,\infty)}
\newcommand{\realPlusPlus}{(0,\,\infty)}
\newcommand{\setXVar}{\mathcal{X}}
\newcommand{\setID}{{\Omega}}
\newcommand{\ambiguity}{\mathcal{W}}
\newcommand{\hatAmbiguity}{\hat{\mathcal{W}}_\mathrm{DR}}
\newcommand{\probabilityHead}{\mathcal{P}}
\newcommand{\probability}{\probabilityHead(\setID)}

\newcommand{\expectation}[2]{\mathbb{E}_{#1}\left[#2\right]}
\newcommand{\variance}[2]{\mathbb{V}_{#1}\left[#2\right]}
\newcommand{\distribution}[2]{\mathbb{P}_{#1}\left[#2\right]}
\newcommand{\VaR}[3]{#1\text{-VaR}_{#2}\left[#3\right]}
\newcommand{\CVaR}[3]{#1\text{-CVaR}_{#2}\left[#3\right]}
\newcommand{\hatCVaR}[3]{#1\text{-}\hat{\text{CVaR}}_{#2}\left[#3\right]}
\newcommand{\tildeCVaR}[3]{#1\text{-}\tilde{\text{CVaR}}_{#2}\left[#3\right]}

\theoremstyle{definition}
\newtheorem{dfn}{Definition}

\newtheorem{prop}[dfn]{Proposition}
\newtheorem{lem}[dfn]{Lemma}
\newtheorem{thm}[dfn]{Theorem}
\newtheorem{cor}[dfn]{Corollary}
\newtheorem{rem}[dfn]{Remark}

\newcommand{\argmin}{\operatornamewithlimits{argmin}}

\title{\LARGE \bf 
\textcolor{black}{
Explicit Reformulation of Discrete Distributionally Robust Optimization Problems
}%textcolor
}

\author{Yuma Shida$^{1,*}$ and Yuji Ito$^{1}$% <-this % stops a space
\thanks{*This paper is submitted to a journal for possible publication. The copyright of this paper may be transferred without notice, after which this version may no longer be accessible. The materials of this paper have been published in part in the conference proceeding \cite{11107818}. This work was not supported by any organization.}% <-this % stops a space
\thanks{$^{1}$Toyota Central R\&D Labs., Inc., 41-1, Yokomichi, Nagakute, Aichi, 480-1192, Japan.
        {\tt\small \{Yuma.shida.fw, ito-yuji\}@mosk.tytlabs.co.jp}%; {\tt\small yuma\_shida@mail.toyota.co.jp}
}%
\thanks{$^{*}$Permanent affiliation: Toyota Motor Corporation, 1, Toyota-Cho, Toyota City, Aichi, 471-8571, Japan. {\tt\small yuma\_shida@mail.toyota.co.jp}
}%
}

\begin{document}

\maketitle
\thispagestyle{empty}
\pagestyle{empty}

%%%%%%%%%%%%%%%%%%%%%%%%%%%%%%%%%%%%%%%%%%%%%%%%%%%%%%%%%%%%%%%%%%%%%%%%%%%%%%%%
\begin{abstract}
\textcolor{black}{
Distributionally robust optimization (DRO) is an effective framework for controlling real-world systems with various uncertainties, typically modeled using distributional uncertainty balls. However, DRO problems often involve infinitely many inequality constraints, rendering exact solutions computationally expensive. 
In this study, we propose a discrete DRO (DDRO) method that significantly simplifies the problem by reducing it to a single trivial constraint. Specifically, the proposed method utilizes two types of distributional uncertainty balls to reformulate the DDRO problem into a single-layer smooth convex program, significantly improving tractability. 
Furthermore, we provide practical guidance for selecting the appropriate ball sizes. 
%: weighted L2 balls and density-ratio balls. 
The original DDRO problem is further reformulated into two optimization problems: one minimizing the mean and standard deviation, and the other minimizing the conditional value at risk (CVaR). These formulations account for  the choice of ball sizes, thereby enhancing the practical applicability of the method. %We derive a interpretable physical value that corresponds to the size of each ball, enhancing the method's applicability. 
The proposed method was applied to a distributionally robust patrol-agent design problem, identifying a Pareto front in which the mean and standard deviation of the mean hitting time varied by up to 3\% and 14\%, respectively, while achieving a CVaR reduction of up to 13\%.
}%textcolor
\end{abstract}

%%%%%%%%%%%%%%%%%%%%%%%%%%%%%%%%%%%%%%%%%%%%%%%%%%%%%%%%%%%%%%%%%%%%%%%%%%%%%%%%
\section{Introduction}
\textcolor{black}{
Real-world systems are exposed to various uncertainties emerging from both natural and societal factors. 
\textcolor{black}{
For example, security robots \cite{duan2021markov,diaz2023distributed} are used for surveillance and protection against threats such as human-caused theft, natural destruction, and  accidents. Two widely used control approaches to manage these uncertainties are stochastic optimal control (SOC) \cite{bertsekas1996stochastic,crespo2003stochastic} and robust control (RC) \cite{sastry2011adaptive,slotine1985robust,scherer2001theory}. 
}%textcolor
SOC minimizes the expected control costs when a reliable stochastic model of uncertainty is available. However, when such a model is difficult to obtain, alternative methods are required. 
Alternatively, RC minimizes the worst-case value of control costs to address broad uncertainty. Nonetheless, its inherent conservativeness can lead to excessively high control costs. 
}%textcolor

%Real-world systems involve several uncertainties. For example, it is difficult for security robots\cite{duan2021markov,diaz2023distributed} to predict the location of a noteworthy event. A popular approach for controlling these uncertain systems is to minimize the control costs with a stochastic model of uncertainty: stochastic optimal control (SOC)\cite{bertsekas1996stochastic,crespo2003stochastic}. If a stochastic model is expensive and/or difficult to obtain, robust control (RC)\cite{sastry2011adaptive,slotine1985robust,scherer2001theory} can address a wide range of uncertainties by considering the worst-case value of the control costs. However, using RC carries the risk of conservative control results.

\textcolor{black}{
Distributionally robust optimization (DRO) has recently emerged as a new approach for enhancing the robustness of control methods and reducing unnecessary conservatism. 
In this method, optimization is typically realized by modeling uncertainties not as the worst-case value, but the worst-case probability distribution within statistical uncertainty sets, often referred to as uncertainty balls. The DRO minimizes the expected value of costs under the worst-case probability distribution \cite{taskesen2024distributionally,liu2023data,yang2020wasserstein,nishimura2021rat,nguyen2023distributionally,coulson2021distributionally}, even if the true distribution of a system is unknown. 
Balls are defined by statistical distances, such as $\phi$-divergence \cite{liu2023data,nishimura2021rat,hu2013kullback} and optimal transport distances \cite{taskesen2024distributionally,yang2020wasserstein,nguyen2023distributionally,pilipovsky2024distributionally,coulson2021distributionally,gao2023distributionally,shafieezadeh2019regularization,cherukuri2022data,luo2017decomposition,mohajerin2018data}. %Even if a true distribution of a system is unknown, DRO minimizes the expected value of the cost function under the worst-case distribution \cite{taskesen2024distributionally,liu2023data,yang2020wasserstein,nishimura2021rat,nguyen2023distributionally,coulson2021distributionally}.
This approach improves the robustness of SOC methods, which often rely on specific assumptions regarding uncertainty, such as the well-known Gaussian noise assumption \cite{nishimura2021rat}. Other problem settings have been explored, including distributionally robust objectives and constraints, particularly those related to value at risk (VaR)  \cite{van2015distributionally,pilipovsky2024distributionally,kishida2022risk}.
}%textcolor

\textcolor{black}{
Discrete DRO (DDRO) is recognized both as a tractable method for discretizing DRO \cite{liu2019discrete}, and as a framework for problems that inherently involve discrete stochastic modeling \cite{farokhi2023distributionally}. 
DRO problems can be addressed using duality principles \cite{mohajerin2018data,gao2023distributionally,yang2020wasserstein,shafieezadeh2019regularization,miao2021data,wiesemann2014distributionally}; however, duality principles often result in  semi-infinite programming (SIP) formulation \cite{reemtsen2013semi,yang2020wasserstein,shafieezadeh2019regularization,cherukuri2022data,mehrotra2014cutting,luo2017decomposition,mohajerin2018data,wiesemann2014distributionally},  which involves infinitely many constraints and renders obtaining exact solutions computationally expensive. 
Discretizing DRO is effective for approximately solving such SIP \cite{liu2023data,reemtsen1991discretization}.  
For example, \cite[Section 5]{mehrotra2014cutting} demonstrates that SIP can be reformulated into linear programming using DDRO methods, such as discretizing probability distributions and using a finite uncertainty set.
%In addition, numerical algorithms for solving SIP directly, such as the cutting-surface method \cite{shafieezadeh2019regularization,cherukuri2022data,mehrotra2014cutting,luo2017decomposition,mohajerin2018data}, have also been developed. These use a finite number of samples, known as cuts, to iteratively approximate the infinitely many constraints \cite{cherukuri2022data}. 
%On a different front, the existing studies \cite{taskesen2024distributionally,liu2023data} have avoided SIP by choosing specific problem settings, such as linear-quadratic settings. 
DDRO is particularly effective when real-world systems have uncertainties represented by discrete distributions, such as categorical sets and finite spaces \cite{farokhi2023distributionally,zhang2021efficient}. For instance, robotic surveillance studies have used discrete modeling of finite locations \cite{duan2021markov,diaz2023distributed}. Existing studies \cite{farokhi2023distributionally,zhang2021efficient} have considered DRO problems using balls defined over discrete distributions, such as the Kantorovich ball and the total variation (TV) ball, the latter being a special case of the optimal transport ball \cite{villani2009optimal}. 
%Our previous study demonstrated that DROC problems with density-ratio (DR) balls, which are considered a subset of the TV balls, can be reformulated into one-layer smooth convex programming, eliminating all inequalities except for a trivial one \cite{11107818}.
}%textcolor

\textcolor{black}{
Previous studies on DDRO have identified two main challenges. 
The existing formulations  \cite{farokhi2023distributionally,zhang2021efficient}   either involve non-trivial inequality constraints or are not expressed as single-layer smooth convex programming. Additionally, in practical applications, determining the appropriate balls remains a significant challenge. 
Developing a theoretical framework that clarifies the effect of ball size would be effective. 
%Understanding the relationship between the choice of these balls and out-of-sample performance \cite{coulson2021distributionally,yang2020wasserstein,shafieezadeh2019regularization,cherukuri2022data,luo2017decomposition,mohajerin2018data,wiesemann2014distributionally,staib2019distributionally,farokhi2023distributionally} is crucial. 
}%textcolor

%In this study, we develop a more tractable formulations of DDRO problems rather than previous studies. Our proposed method reformulates min-max optimization problems commonly encountered in DDRO into single-layer smooth convex programming associated with two types of balls: weighted L2 balls and density-ratio (DR) balls. Our proposed method provides insights into how to determine both the type and size of these balls. Specifically, we show that choosing the weighted L2 balls in DDRO is equivalent to minimizing a weighted sum of the expected cost and its standard deviation, and choosing the DR balls in DDRO is equivalent to minimizing conditional VaR (CVaR) \cite{rockafellar2000optimization}. 

\textcolor{black}{
%Our contributions include solvability, interpretability, and demonstrations. 
%Unlike our conference paper \cite{ shida2024discretedistributionallyrobustoptimal}, this study newly focuses on reformulating single-layer smooth convex programming with trivial constraints associated with the weighted L2 balls. 
In this study, we developed a more tractable  formulation of DDRO problems than those used in previous studies. Our proposed method reformulates min-max optimization problems in DDRO into single-layer smooth convex programming with trivial constraints associated with weighted L2 and density-ratio (DR) balls. 
Additionally, we derive physically interpretable values that provide insights into how to determine the ball size. 
Specifically, we demonstrate that choosing weighted L2 balls in DDRO is equivalent to minimizing the weighted sum of the expected cost and its standard deviation and choosing the DR balls in DDRO is equivalent to minimizing conditional VaR (CVaR) \cite{rockafellar2000optimization}. 
%In contrast, the conference paper \cite{ shida2024discretedistributionallyrobustoptimal} have only addressed reformulating single-layer smooth convex programming with trivial constraints associated with the DR balls. 
The main contributions of this study are summarized as follows.
\begin{itemize}
\item Solvability: The proposed method for solving DDRO problems with weighted L2 balls and DR balls can be reformulated into  single-layer smooth convex programming with trivial constraints rendering them solvable. These results are presented in Theorems \ref{thm:strongDualL} and \ref{thm:strongDual} in Section \ref{sec:mainResult2}. 
\item Interpretability: Our proposed method clarifies how the size of the ball affects DDRO based on physical values related to the weighted L2 and DR balls. 1) Minimizing expectation and standard deviation: We demonstrate that solving DDRO problems associated with weighted L2 balls is equivalent to minimizing the weighted sum of the expected cost and its standard deviation. This aligns with conventional control theories, such as risk-sensitive control \cite{jacobson2003optimal, whittle1981risk}, which are compatible with minimizing both the expectation and higher-order moments, such as standard deviation. 
2) Minimizing CVaR: We show that solving DDRO problems associated with DR balls is equivalent to minimizing the CVaR of the control cost function. These results are shown in Theorems \ref{thm:variance}, \ref{thm:CVaR}, and \ref{thm:knapsack}, and Corollary \ref{cor:variance} in Section \ref{sec:mainResult1}.
%\item Desidability: Our proposed method helps clarify how the size of the ball affects DDROC, based on physical values related to the weighted L2 balls and the DR balls. These insights also imply that the choice of balls depends on the distribution shape of the control cost function. 
\item Demonstration: We demonstrate that the proposed method can be solved as a general convex programming problem through numerical experiments on patroller-agent design problems from \cite{diaz2023distributed}. This design is adapted to fit the DDRO framework. 
%\end{itemize}
%This paper's main extensions are summarized below.
%\begin{itemize}
\end{itemize}
}%textcolor

\textcolor{black}{
This study is an extended version of our previous conference paper\cite{11107818}. It investigates the solvability issues related to weighted L2 balls and enhances the interpretability of weighted L2 and DR balls. 
In contrast, the conference paper\cite{11107818} only addressed the solvability issues related to DR balls. 
}%textcolor

\section*{Notation}

\textcolor{black}{
We use the following notations:
\begin{itemize}
\item $\eye_a$: Identity matrix of size $a\times a$.
\item $[\anyVector]_j$: $j$-th component of a vector $\anyVector\in\real^a$.
\item $\diag(\anyVector) \coloneqq 
\begin{bmatrix}
[\anyVector]_{1} & & 0 \\
 & \ddots & \\
0 & & [\anyVector]_{a}
\end{bmatrix}
$: Diagonal matrix formed from the components of a vector $\anyVector\in\real^{a}$.
\item $[\anyMatrix]_{j,\,k}$: Element in the $j$-th row and $k$-th column of a matrix $\anyMatrix\in\real^{a\times b}$.
\item $\vect(\anyMatrix) \coloneqq
[
[\anyMatrix]_{1,\,1}  \cdots  [\anyMatrix]_{a,\,1}  \cdots  [\anyMatrix]_{1,\,b}  \cdots  [\anyMatrix]_{a,\,b} 
]
^\top$: Vectorization of a matrix $\anyMatrix\in\real^{a\times b}$, stacking its columns into a single vector. 
\item $\relint(\anySet)$: Relative interior of a set $\anySet\subseteq\real^a$.
\item 
\textcolor{black}{
$\displaystyle\probabilityHead(\anySet) \coloneqq \{\realProbHead:\anySet\rightarrow[0,\,1]\mid\sum_{\anyRandom\in\anySet}\realProbHead(\anyRandom)=1\}$: Set of all probability mass functions of a discrete random variable $\anyRandom\in\anySet$ over a finite set $\anySet$.
}%textcolor
\item 
\textcolor{black}{
$\expectation{\realProbHead(\anyRandom)}{\anyFcn(\anyRandom)}$: Expectation of $\anyFcn(\anyRandom)$ with respect to a random variable $\anyRandom$ under a distribution $\realProbHead(\anyRandom)$.
}%textcolor
\item 
\textcolor{black}{
$\expectation{\realProbHead(\anyRandom)}{\anyFcn(\anyRandom)\mid\anyEvent}$: Conditional expectation of $\anyFcn(\anyRandom)$ 
given that $\anyEvent$ holds. 
}%textcolor
\item 
\textcolor{black}{
$\variance{\realProbHead(\anyRandom)}{\anyFcn(\anyRandom)} \coloneqq \expectation{\realProbHead(\anyRandom)}{(\anyFcn(\anyRandom)-\expectation{\realProbHead(\anyRandom)}{\anyFcn(\anyRandom)})^2}$: Variance of $\anyFcn(\anyRandom)$ under a distribution $\realProbHead(\anyRandom)$.
}%textcolor
\item 
\textcolor{black}{
$\distribution{\realProbHead(\anyRandom)}{\anyRandom\in\anySet}$: Probability that $\anyRandom\in\anySet$ under a distribution $\realProbHead(\anyRandom)$. If the distribution used is clear, we note it as $\distribution{}{\anyRandom\in\anySet}$.
}%textcolor
\item 
\textcolor{black}{
$\distribution{\realProbHead(\anyRandom)}{\anyRandom_1\in\anySet_1\mid\anyRandom_2\in\anySet_2}$: Conditional probability that $\anyRandom_1\in\anySet_1$ given $\anyRandom_2\in\anySet_2$ under a distribution $\realProbHead(\anyRandom)$. If the distribution used is clear, we note it as $\distribution{}{\anyRandom_1\in\anySet_1\mid\anyRandom_2\in\anySet_2}$. 
}%textcolor
\item 
\textcolor{black}{
$\VaR{\beta}{\realProbHead(\anyRandom)}{\anyFcn(\anyRandom)} \coloneqq \inf\{\alpha\in\real\mid\distribution{\realProbHead(\anyRandom)}{\anyFcn(\anyRandom)\le\alpha}\ge\beta\}$: VaR of $\anyFcn(\anyRandom)$ at level $\beta\in[0,1]$ %of $\anyRandom$ 
under a distribution $\realProbHead(\anyRandom)$.
}%textcolor
\item 
\textcolor{black}{
$\CVaR{\beta}{\realProbHead(\anyRandom)}{\anyFcn(\anyRandom)} \coloneqq \expectation{\realProbHead(\anyRandom)}{\anyFcn(\anyRandom)\mid\anyFcn(\anyRandom)\ge\VaR{\beta}{\realProbHead(\anyRandom)}{\anyFcn(\anyRandom)}}$: CVaR of $\anyFcn(\anyRandom)$ at level $\beta\in[0,1]$ %of $\anyRandom$ 
under a distribution $\realProbHead(\anyRandom)$.
}%textcolor
\end{itemize}
}%textcolor

\section{Target Systems and Problems Setting}

%\subsection{Target Systems}
\textcolor{black}{
We consider a target system that involves a decision variable $\xVar\in\setXVar\subseteq\real^\dimXVar$ in a particular set $\setXVar$ and a random variable $\ID\in\setID$, where  the set is $\setID=\{1,\,2,\,\cdots,\,m\}$. The probability distribution $\realProbHead\in\probability$ of $\ID$ is assumed to be unknown but lies within a ball $\ambiguity\subseteq\probability$. The performance of the system is evaluated based on the expectation of a cost function $\targetFcn$, denoted as $\expectation{\realProb}{\targetFcn}$. The cost function is defined as $\targetFcnHead:\setXVar\times\setID\to\mathbb{R}$ and represents the objective to be minimized. For example, it may represent the time required to complete a process in an industrial application, such as robotic control. The definition of the cost function indicates that for each $\xVar$, $\expectation{\realProb}{\targetFcn} < \infty$. This study considers the following problem.
}%textcolor

%\subsection{Discrete Distributionally Robust Optimal Control Problems}

\textcolor{black}{
\textit{DDRO problem:} Design a decision variable that minimizes the worst-case expectation of the cost function of the target system within a given ball. 
\begin{equation}\label{eq:DROC}
\min_{\xVar\in\setXVar}\quad\innerMax.
\end{equation}
Here, the ball $\ambiguity$ is defined in Section \ref{sec:method}. % $\ambiguity=\ambiguityL$ or $\ambiguity=\ambiguityDR$. 
%\begin{equation}\label{eq:ball}
%\ambiguity\in\{\ambiguityL,\,\ambiguityDR\}.
%\end{equation}
\begin{rem}[Difficulty in solving DDRO problems]\label{rem:difficulty}
The DDRO problem in (\ref{eq:DROC}) is a two-layer min-max optimization problem and not a single-layer optimization problem. Directly solving this min-max optimization problem remains challenging. One approach to solving this problem is to reformulate the inner maximization problem as a minimization problem \cite{farokhi2023distributionally}. 
%It demonstrates that this min-max optimization problem can be reformulated as a single-layer min-min problem by duality principles. 
However, the study in \cite{farokhi2023distributionally} showed that non trivial constraints remain. 
These constraints are not infinite, but rather finite many, yet sufficient to scale with the size of the support set of the probability distribution. 
Another approach is to solve the problem as a saddle-point problem rather than as a single-layer optimization problem \cite{zhang2021efficient}. The study proposes a new algorithm that finds a saddle point and guarantees an $\mathcal{O}(1/\,\epsilon)$ iteration complexity. Here, $\epsilon>0$ is the required accuracy. However, this complexity is less efficient compared to the $\mathcal{O}(\log(1/\,\epsilon))$ iteration complexity typically achieved in convex optimization \cite{boyd2004convex}. 
%, especially when high-accuracy solutions are required. 
Furthermore, existing studies \cite{farokhi2023distributionally,zhang2021efficient}  have focused primarily on specific types of balls: the Kantorovich ball and TV ball, which may limit the generality of their approaches.
%This DDROC problem is reformulated as the following SIP with respect to a new variable $t\in\real$:
%\begin{equation}\label{eq:DROC}
%\begin{split}
%&\min_{\xVar\in\setXVar,\,t\in\real}\quad t, \\
%&\text{s.t.}~t\ge\inner, \quad \forall \prob\in\ambiguity.
%\end{split}
%\end{equation}
%However, solving SIP exactly is costly. In order to reduce the solving cost, it is effective to consider other reformulation methods. 
\end{rem}
}%textcolor 

\section{Proposed Method}\label{sec:method}

\textcolor{black}{
To address the challenges in solving min-max optimization problems in Remark \ref{rem:difficulty}, we propose explicit reformulations that transform them into single-layer smooth convex optimization problems. We consider two types of balls: a weighted L2 ball, $\ambiguity=\ambiguityL$ and a DR ball, $\ambiguity=\ambiguityDR$. These balls are defined as follows:
\begin{equation}\label{eq:ambiguityL}
\ambiguityL = \{\probHead\in\probability\mid\sqrt{\expectation{\refProb}{(\densRatio-1)^2}}\le\const\}, 
\end{equation}
\begin{equation}\label{eq:ambiguity}
\ambiguityDR = \{\probHead\in\probability\mid\forall \ID\in\setID, ~ \densRatio\le1+\const\}. 
\end{equation}
Here, $\densRatio\coloneqq\prob/\,\refProb\in\realPlus$ denotes the density ratio between a reference distribution $\refProbHead$ and any candidate probability distribution $\probHead\in\probability$. The reference distribution $\refProbHead\in\probability$ is a probability distribution centered within the balls. We assume that the center within the balls satisfies $\refProb>0$ for all $\ID\in\setID$. The positive constant $\const>0$ is employed to control the size of the balls. 
\begin{rem}[Motivation for Using Weighted L2 and Density-Ratio Balls]
The weighted L2 and DR balls in (\ref{eq:ambiguityL}) and (\ref{eq:ambiguity}) have properties that enhance the tractability of the DDRO problem. Specifically, they can be characterized by constraints using  differentiable and strictly convex functions. These properties are discussed in detail in Section \ref{sec:subResult1}.   
\end{rem}
}%textcolor 

\begin{figure*}[t]
   \centering
   %\framebox{\parbox{3in}{We suggest that you use a text box to insert a graphic (which is ideally a 300 dpi TIFF or EPS file, with all fonts embedded) because, in an document, this method is somewhat more stable than directly inserting a picture.}}
   \includegraphics[scale=0.38,clip]{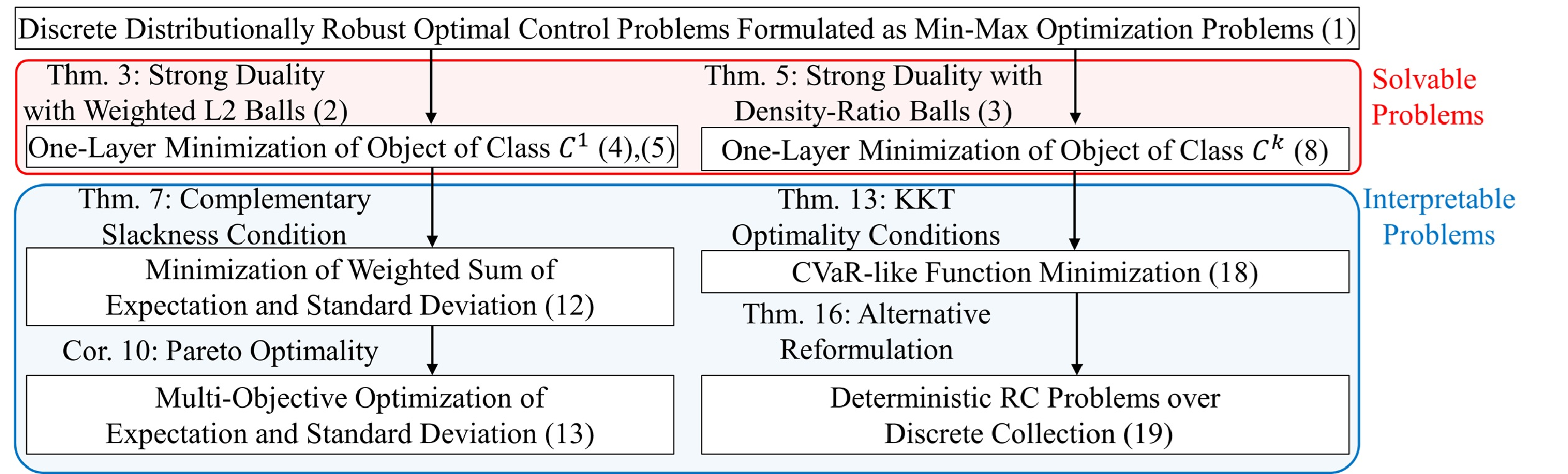}
   \caption{Overview of the proposed method.}
   \label{fig:overview}
\end{figure*}

\textcolor{black}{
An overview of the proposed method is shown in Fig. \ref{fig:overview}. 
Theorems \ref{thm:strongDualL} and \ref{thm:strongDual} in Section \ref{sec:mainResult2} demonstrate that the DDRO problem in (\ref{eq:DROC}) can be reformulated as a single-layer minimization problem with only trivial constraints. When solving the DDRO problem associated with each defined ball, it is essential to understand the effect of each ball size $\const$. Theorems \ref{thm:variance}, \ref{thm:CVaR}, \ref{thm:knapsack}, and Corollary \ref{cor:variance} in Section \ref{sec:mainResult1} provide theoretical insights into how this size can be determined. Detailed proofs of these theorems are presented in Sections \ref{sec:subResult1} and \ref{sec:subResult2}. 
}%textcolor

\subsection{
Solvability Results of the DDRO Problems
}\label{sec:mainResult2}

\textcolor{black}{
We introduce the Lagrange multipliers associated with the weighted L2 ball, $\ineqVar\in\realPlus$ and $\eqVar\in\real$.  Using these, we consider the following  single-layer minimization problems:  
\begin{equation}\label{eq:dualDROCL}
\min_{\xVar\in\setXVar}\quad\inf_{\KKTVar\in\realPlus\times\real}\innerDualL, 
\end{equation}
\begin{equation}\label{eq:dualDROCRealL}
\min_{\xVar\in\setXVar}\quad\inf_{\KKTVar\in\realPlusPlus\times\real}\dualL. 
\end{equation}
Here, $\dualExtendedHeadL :\setXVar\times\realPlus\times\real\rightarrow\real\cup\{\infty\}$ and $\dualHeadL :\setXVar\times\realPlusPlus\times\real\rightarrow\real$ are denoted as the following functions:
\begin{multline}\label{eq:zeroWorstL}
\dualExtendedL \\ \coloneqq 
\begin{cases}
\eqVar, & (\ineqVar=0, ~ \max_{\ID\in\setID}\targetFcn\le\eqVar), \\
\dualL, & (\ineqVar\ne0),\\
\infty, & (\ineqVar=0, ~ \max_{\ID\in\setID}\targetFcn>\eqVar),
\end{cases}
\end{multline}
\begin{multline}\label{eq:worstL}
\dualL \coloneqq \ineqVar\,\expectation{\refProb}{\maxFcn{0,~\frac{\targetFcn+2\ineqVar-\eqVar}{2\ineqVar}}^2}\\+\ineqVar\,(\const^2-1) +\eqVar.
\end{multline}
\begin{thm}[Reformulation of DDRO Problems with Weighted L2 Balls]\label{thm:strongDualL}
Problems in (\ref{eq:dualDROCL}) and (\ref{eq:dualDROCRealL}) satisfy the following properties:
\begin{enumerate}
\renewcommand{\labelenumi}{(\roman{enumi})}
\item Minimizers of $\xVar$ to (\ref{eq:dualDROCL}) are equivalent to those to the DDRO problem in (\ref{eq:DROC}) associated with the weighted L2 ball $\ambiguity=\ambiguityL$ in (\ref{eq:ambiguityL}). 
\item Suppose that $\max_{\ID\in\setID}\targetFcnHead(\xVar^*,\,\ID)>\expectation{\refProb}{\targetFcnHead(\xVar^*,\,\ID)}+\const\sqrt{\variance{\refProb}{\targetFcnHead(\xVar^*,\,\ID)}}$ and $\min_{\ID\in\setID}\targetFcnHead(\xVar^*,\,\ID)>\expectation{\refProb}{\targetFcnHead(\xVar^*,\,\ID)}-\sqrt{\variance{\refProb}{\targetFcnHead(\xVar^*,\,\ID)}}/~\const$ are satisfied for a minimizer $\xVar^*$ to  (\ref{eq:DROC}). Subsequently, $\xVar^*$ is equivalent to a minimizer to the problem in (\ref{eq:dualDROCRealL}). 
\item If the cost function $\targetFcn$ is strictly convex and continuous on a bounded closed convex set $\setXVar$ for each $\ID\in\setID$, the minimizer of $\xVar$ corresponding to each of (\ref{eq:dualDROCL}) and  (\ref{eq:dualDROCRealL}) is unique.
\item If the cost function $\targetFcn$ is convex and of class $C^1$ on an open convex set $\setXVar$ for each $\ID\in\setID$, the objective functions of (\ref{eq:dualDROCL}) and (\ref{eq:dualDROCRealL}), $\innerDualL$ and $\dualL$, respectively, are also convex and of class $C^1$ on $\setXVar\times\realPlusPlus\times\real$.
\end{enumerate}
\end{thm}
\begin{rem}[Solvability of DDRO Problems with Weighted L2 balls]
Theorem \ref{thm:strongDualL} (i) states that minimizers to the single-layer minimization problem in (\ref{eq:dualDROCL}) are equivalent to the solutions to the DDRO problem in (\ref{eq:DROC}). 
Furthermore, if minimizers to (\ref{eq:dualDROCL}) that satisfy the conditions in Theorem \ref{thm:strongDualL} (ii) exist, they are strictly equivalent to solutions to a single-layer minimization problem with trivial constraints in (\ref{eq:dualDROCRealL}). 
If the cost function $\targetFcn$ is convex and continuously differentiable on an open  convex set $\setXVar$ for each $\ID\in\setID$, Theorem \ref{thm:strongDualL} (iii) and (iv) state that the problem in (\ref{eq:dualDROCL}) becomes a smooth convex optimization problem. This can be solved using general gradient-based algorithms, such as the interior point method \cite{boyd2004convex}.
As a side note, the problem in (\ref{eq:dualDROCL}) includes only trivial constraints for all $\ineqVar>0$ but $\ineqVar=0$. However, when $\ineqVar=0$, it implicitly includes constraints $\max_{\ID\in\setID}\targetFcn\le\eqVar$. 
Meanwhile, $\dualL$ closely approximates $\dualExtendedL$, as $\lim_{\ineqVar\rightarrow0^+}\dualL=\dualExtendedHeadL(\xVar,\,0,\,\eqVar)$. From a practical perspective, this enables us to approximate the problem in (\ref{eq:dualDROCL}) as in (\ref{eq:dualDROCRealL}). 
Theorem \ref{thm:strongDualL} only states that $\innerDualL$ is of class $C^1$, even though $\targetFcn$ is of class $C^k$ for some $k\ge1$. This limitation arises because $\innerDualL$ includes quadratic terms involving the max function.
\end{rem}
}%textcolor

\textcolor{black}{
Subsequently, we introduce the Lagrange multipliers associated with the DR ball, $\ineqVars\in\realPlus^\dimID$ and $\eqVar\in\real$. Using these, we consider the following optimization problem:
\begin{equation}\label{eq:dualDROC}
\min_{\xVar\in\setXVar}\quad\inf_{\KKTVars\in\realPlus^\dimID\times\real}\innerDual.
\end{equation}
Here, $\dualExtendedHead :\setXVar\times\setID\times\realPlus\times\real\rightarrow\real\cup\{\infty\}$ is denoted as the following functions:
\begin{multline}\label{eq:zeroWorst}
\dualExtended \\ \coloneqq
\begin{cases}
\eqVar, & ([\ineqVars]_\ID=0, ~ \targetFcn\le\eqVar), \\
\dual, & ([\ineqVars]_\ID\ne0),\\
\infty, & ([\ineqVars]_\ID=0, ~ \targetFcn>\eqVar), 
\end{cases}
\end{multline}
\begin{multline}\label{eq:worst}
\dual \\ \coloneqq (1+\const)\,[\ineqVars]_\ID\expFcn{\frac{\targetFcn-[\ineqVars]_\ID-\eqVar}{[\ineqVars]_\ID}}+\eqVar.
\end{multline}
\begin{thm}[Reformulation of DDRO Problems with Density-Ratio Balls]\label{thm:strongDual}
The problem in (\ref{eq:dualDROC}) satisfies the following properties:
\begin{enumerate}
\renewcommand{\labelenumi}{(\roman{enumi})}
\item Minimizers of $\xVar$ to (\ref{eq:dualDROC}) are equivalent to those to the DDRO problem in (\ref{eq:DROC}) associated with the DR ball $\ambiguity=\ambiguityDR$ in  (\ref{eq:ambiguity}).
\item If the cost function $\targetFcn$ is strictly convex and continuous on a bounded closed convex set $\setXVar$ for each $\ID\in\setID$, the minimizer of $\xVar$ to (\ref{eq:dualDROC}) is unique.
\item If the cost function $\targetFcn$ is convex and of class $C^k$ on an open convex set $\setXVar$ for each $\ID\in\setID$, the objective function of (\ref{eq:dualDROC}) is also convex and of class $C^k$ on $\setXVar\times\realPlusPlus^\dimID\times\real$.
\end{enumerate}
\end{thm}
\begin{rem}[Solvability of DDRO Problems with DR balls]
Theorem \ref{thm:strongDual} (i) states that minimizers to a single-layer minimization problem in (\ref{eq:dualDROC}) is equivalent to the solutions to the DDRO problem in (\ref{eq:DROC}). 
If the cost function $\targetFcn$ is convex and continuously differentiable on an open  convex set $\setXVar$ for each $\ID\in\setID$, Theorem \ref{thm:strongDual} (ii) and (iii) states that the problem in (\ref{eq:dualDROC}) also become a smooth convex optimization problem. This can be solved using general gradient-based algorithms, such as the interior point method \cite{boyd2004convex}. 
As a side note, the problem in (\ref{eq:dualDROC}) includes only trivial constraints when $[\ineqVars]_\ID>0$ for all $\ID\in\setID$. However, when there exists $\ID\in\setID$ which satisfies $[\ineqVars]_\ID=0$, it implicitly includes the constraint $\targetFcn\le\eqVar$. 
Meanwhile, $\dual$ closely approximates $\dualExtended$, as $\lim_{[\ineqVars]_\ID\rightarrow0^+}\dual=\dualExtendedHead(\xVar,\,\ID,\,0,\,\eqVar)$. From a practical perspective, this enables us to approximate the problem in (\ref{eq:dualDROC}) as follows:
\begin{equation}
\min_{\xVar\in\setXVar}\quad\inf_{\KKTVars\in\realPlusPlus^\dimID\times\real}\expectation{\refProb}{\dual}.
\end{equation}
\end{rem}
}%textcolor

\subsection{
Interpretability Results of the Size of Balls
}\label{sec:mainResult1}
\textcolor{black}{
To reformulate the DDRO problem into an interpretable formulation, a standard deviation-based problem corresponding to the DDRO problem associated with a weighted L2 ball is considered. We introduce the following problem that minimizes the weighted sum of the expectation and standard deviation:
\begin{equation}\label{eq:variance}
\min_{\xVar\in\setXVar} \quad \expectation{\refProb}{\targetFcn}+\const\sqrt{\variance{\refProb}{\targetFcn}}.
\end{equation}
\begin{thm}[Expectation and Standard Deviation Minimization]\label{thm:variance}
Suppose that $\max_{\ID\in\setID}\targetFcnHead(\xVar^*,\,\ID)>\expectation{\refProb}{\targetFcnHead(\xVar^*,\,\ID)}+\const\sqrt{\variance{\refProb}{\targetFcnHead(\xVar^*,\,\ID)}}$ and $\min_{\ID\in\setID}\targetFcnHead(\xVar^*,\,\ID)>\expectation{\refProb}{\targetFcnHead(\xVar^*,\,\ID)}-\sqrt{\variance{\refProb}{\targetFcnHead(\xVar^*,\,\ID)}}/~\const$ are satisfied for a minimizer $\xVar^*$ to the DDRO problem in (\ref{eq:DROC}) associated with the weighted L2 ball $\ambiguity=\ambiguityL$ in (\ref{eq:ambiguityL}). This minimizer is equivalent to those to (\ref{eq:variance}). 
%\begin{enumerate}
%\renewcommand{\labelenumi}{(\roman{enumi})}
%\item Minimizer of $\xVar$ to the following problem: 
%\begin{equation}
%\begin{split}
%\min_{\xVar\in\setXVar}\quad\min\Bigg\{&\max_{\ID\in\setID} \targetFcn,~ \\&\expectation{\refProb}{\targetFcn} +\const\sqrt{\variance{\refProb}{\targetFcn}}\Bigg\},
%\end{split}
%\end{equation}
%is equivalent to those to the DDROC problem in (\ref{eq:DROC}) associated with the weighted L2 ball $\ambiguity=\ambiguityL$ in  (\ref{eq:ambiguityL}).
%\end{enumerate}
\end{thm}
\begin{rem}[Weight Parameters and Size of Weighted L2 Balls]
%Theorem \ref{thm:variance} (i) implies that the DDROC problem associated with the weighted L2 ball minimizes the lesser of two objectives: 1) the weighted sum of the expectation and the square root of the variance, or 2) the maximum of the cost function. 
Theorem \ref{thm:variance} implies that the size of the weighted L2 ball denoted in (\ref{eq:ambiguityL}), $\const$, corresponds to the weight parameter of the problem in (\ref{eq:variance}) provided $\const$ is small. This selects the trade-off between the average performance and its variability.
\end{rem}
}%textcolor

\textcolor{black}{
Minimizers to the problem in (\ref{eq:variance}) provide Pareto-optimal solutions for multi-objective optimization \cite[Section 2.1]{pardalos2017non}. Therefore, this enables us to understand the trade-off chosen by $\const$. Consider a multi-objective optimization problem in the following form:
\begin{equation}\label{eq:multi}
\min_{\xVar\in\setXVar} \quad \Bigg\{\expectation{\refProb}{\targetFcn},~\sqrt{\variance{\refProb}{\targetFcn}}\Bigg\}.
\end{equation}
\begin{rem}[Pareto Front and Optimality]
The selection of the weight parameter $\const$ provides an element of the set of Pareto-optimal solutions called the Pareto front \cite[Section 2.1]{pardalos2017non}. 
The optimality of these Pareto-optimal solutions to (\ref{eq:multi}) is defined as follows \cite[Definition 1.3]{pardalos2017non}: A point $\xVar^*\in\setXVar$ is a Pareto-optimal solution if there is no $\xVar\in\setXVar$ that satisfies either of the following equations:
\begin{multline}\label{eq:largeExpect}
\expectation{\refProb}{\targetFcn}<
\expectation{\refProb}{\targetFcnHead(\xVar^*,\,\ID)},\quad \\  
\sqrt{\variance{\refProb}{\targetFcn}}\le  
\sqrt{\variance{\refProb}{\targetFcnHead(\xVar^*,\,\ID)}},
\end{multline}
\begin{multline}\label{eq:largeVariance}
\expectation{\refProb}{\targetFcn}\le
\expectation{\refProb}{\targetFcnHead(\xVar^*,\,\ID)},\quad \\  
\sqrt{\variance{\refProb}{\targetFcn}}<  
\sqrt{\variance{\refProb}{\targetFcnHead(\xVar^*,\,\ID)}}.
\end{multline} 
\end{rem}
\begin{cor}[Pareto-optimal solutions to Expectation and Standard Deviation Minimization]\label{cor:variance}
Minimizers to (\ref{eq:variance}) are some  Pareto-optimal solutions to (\ref{eq:multi}). 
\end{cor}
\begin{rem}[Difficulty in Minimizing Standard Deviation]
Although the formulations of the problems in (\ref{eq:variance}) and (\ref{eq:multi}) are clear, they may be difficult to solve directly. This is because the standard deviation is not necessarily a convex function of $\xVar$. Therefore, solving the equivalent DDRO problem associated with the weighted L2 ball in  (\ref{eq:dualDROCL}) is often more effective than directly solving (\ref{eq:variance}) or (\ref{eq:multi}). 
\end{rem}
}%textcolor

\textcolor{black}{
To reformulate the DDRO problem into another interpretable formulation, we consider a CVaR-based problem corresponding to the DDRO problem associated with the DR ball. Let a CVaR-like function be associated with a probability level $\beta\in[0,\,1]$. We denote this function as $\beta\text{-}\hat{\text{CVaR}}$:  
\begin{multline}\label{eq:hatCVaR}
\targetHatCVaR \\ \coloneqq \expectation{\refProb}{\targetFcn\mid\targetFcn>\targetVaR}.
\end{multline}
The function $\beta\text{-}\hat{\text{CVaR}}$ becomes equivalent to $\beta\text{-}\text{CVaR}$ when the strict inequality in (\ref{eq:hatCVaR}) is replaced with a non-strict inequality.
Consider a $\beta\text{-}\tilde{\text{CVaR}}$ function that satisfies the following equation as
\begin{multline}\label{eq:tildeCVaR}
\targetCVaR\le\targetTildeCVaR \\ \le\targetHatCVaR.
\end{multline}
Futhermore, we introduce the $\beta\text{-}\tilde{\text{CVaR}}$ minimization problem as follows:
\begin{equation}\label{eq:CVaR}
\min_{\xVar\in\setXVar} \quad \targetTildeCVaR.
\end{equation}
Here, we consider the $\beta\text{-}\tilde{\text{CVaR}}$ minimization problem with some $\beta\text{-}\tilde{\text{CVaR}}$ function that satisfies (\ref{eq:tildeCVaR}).
\begin{rem}[$\beta\text{-}\hat{\text{CVaR}}$ and $\beta\text{-}\tilde{\text{CVaR}}$]
The definitions of 
$\beta\text{-}\hat{\text{CVaR}}$ and $\beta\text{-}\tilde{\text{CVaR}}$ differ from $\beta$-CVaR in previous studies \cite[Theorem 1] {rockafellar2000optimization},\cite[Proposition 5.11]{kuhn2024distributionallyrobustoptimization} because $\refProb$ is a discrete distribution; namely, some $\alpha\in\real$ satisfies $\sum_{\ID\in\{\ID\in\setID\mid\targetFcn=\alpha\}}\refProb\ne0$. 
\end{rem}
\begin{thm}[Conditional Value at Risk Minimization]\label{thm:CVaR}
Provided that the probability level $\beta\coloneqq\const/\,(1+\const)$, the problem in (\ref{eq:CVaR}) satisfies the following properties: 
\begin{enumerate}
\renewcommand{\labelenumi}{(\roman{enumi})}
%Then, 
\item Minimizers of $\xVar$ to some $\beta\text{-}\tilde{\text{CVaR}}$ function in (\ref{eq:CVaR}) are equivalent to those to the DDRO problem in (\ref{eq:DROC}) associated with the DR ball $\ambiguity=\ambiguityDR$ in  (\ref{eq:ambiguity}).
\item %There exist $\beta\text{-}\tilde{\text{CVaR}}$ which satisfies (\ref{eq:tildeCVaR}) and 
If there exists $\alpha\in\real$ such that $\distribution{\refProb}{\targetFcn\ge\alpha}=\beta$, the $\beta\text{-}\tilde{\text{CVaR}}$ in (\ref{eq:CVaR}) is uniquely determined as $\targetTildeCVaR=\targetHatCVaR$. 
\end{enumerate}
\end{thm}
\begin{rem}[Probability Level and Size of Density-Ratio Balls]
Theorem \ref{thm:CVaR} indicates that the probability level $\beta=\const/(1+\const)$ corresponds to the size of the DR ball defined in (\ref{eq:ambiguity}). The probability level $\beta$ monotonically increases with respect to $\const$. Theorem \ref{thm:CVaR} (ii) states that the specific formulation of $\beta\text{-}\tilde{\text{CVaR}}$ can be described as $\beta\text{-}\hat{\text{CVaR}}$.
\end{rem} 
}%textcolor

\textcolor{black}{
\begin{rem}[Comparison of Balls]
The size of the weighted L2 ball reflects the trade-off parameter $\const$, which balances the average performance and its variability. In contrast, the size of the DR ball is determined by the probability level $\beta$, which corresponds to the threshold for evaluating  the maximum cost with probability greater than $\beta$. If the focus is on overall performance, the weighted L2 ball is appropriate. If the focus is on the worst-case cost, a DR ball should be used.
\end{rem}
}%textcolor

\textcolor{black}{
We present an alternative interpretation of the DDRO problem in (\ref{eq:DROC}) associated with the DR ball in (\ref{eq:ambiguity}) as a deterministic RC problem over a discrete collection using the following theorem. 
%\begin{rem}[Alternative Interpretation of $\beta\text{-}\tilde{\text{CVaR}}$]
%The deterministic RC control problem over a discrete collection is useful to provide the specific formulation of some $\beta\text{-}\tilde{\text{CVaR}}$ functions in (\ref{eq:CVaR}). 
%\end{rem}
\begin{thm}[Deterministic RC Problems with Worst $\constC$ Costs]\label{thm:knapsack}
Assume that the reference distribution is uniform, as $\refProb=1/\,\dimID$ and $\constC\coloneqq\dimID/\,(1+\const)$ is a positive integer. Subsequently, minimizers of $\xVar$ to (\ref{eq:DROC}) associated with (\ref{eq:ambiguity}) are equivalent to those to the following deterministic RC problem:
\begin{equation}\label{eq:knapsackRC2}
\min_{\xVar\in\setXVar}\quad
\max_{(\ID_1,\,\cdots,\,\ID_{\constC})\in\ambiguityC}\quad 
\sum_{l=1}^{\constC}\targetFcnHead(\xVar,\,\ID_l),
\end{equation}
\begin{multline}\label{eq:worstC}
\ambiguityC \coloneqq \{(\ID_1,\,\cdots,\,\ID_\constC)\in\setID^\constC \mid\forall j\in\{1,\cdots,\constC\},~ \\ 
\forall k\in\{1,\cdots,\constC\}\setminus \{j\},~\ID_j\ne\ID_k\}.
\end{multline}  
\end{thm}
%\begin{rem}[Proof Idea of Theorem \ref{thm:knapsack}]
%Theorem \ref{thm:knapsack} is proved immediately by Theorem \ref{thm:CVaR}. We provide the proof of Theorem \ref{thm:knapsack} in this section as follows.
%\end{rem}
%\begin{proof}[Proof of Theorem \ref{thm:knapsack}]
%From the assumption that $\constC=\dimID/\,(1+\const)$ is a positive integer, $1 / \, (1+\const)$ is a positive integer multiple of $1/\,\dimID$. In addition, from the result that $1 / \, (1+\const)\le1$, $\const / \, (1+\const)$ is non-negative integer multiple of $1/\,\dimID$. Additionally, from the result that $\refProb=1/\,\dimID$ and $\targetFcnHead(\xVar,\,j)\ne\targetFcnHead(\xVar,\,k)$ is satisfied for all $j\ne k$, there are $\alpha\in\real$ such that $\distribution[\targetFcn\ge\alpha]=\beta=\const/\,(1+\const)$. Therefore, from Theorem \ref{thm:CVaR}, the DDROC probelm (\ref{eq:DROC}) associated with (\ref{eq:ambiguity}) is equivalent to $\targetHatCVaR$ minimization problem in (\ref{eq:CVaR}). In this case, definitions of $\targetHatCVaR$ is equivalent to conditional expectation of worst $\constC$ cost functions; namely:
%\begin{equation*}
%\targetHatCVaR=\innerMaxKnapsackUniform.
%\end{equation*}
%That yields the statement. Also see in \cite[Theorem 4]{11107818}.
%\end{proof}
}%textcolor

\subsection{Proofs of Theorems \ref{thm:strongDualL} and \ref{thm:variance}}
\label{sec:subResult1}
\textcolor{black}{
We prove Theorems \ref{thm:strongDualL} and \ref{thm:variance} after establishing Lemmas \ref{lem:dual} and \ref{lem:lagrangeDualL}, respectively. We consider the following ball: 
\begin{equation}\label{eq:someBall}
\ambiguity=\{\probHead\in\probability\mid \forall j\in\{1,\,\cdots,\,b\},~f_j(\densRatioHead_1,\,\cdots,\,\densRatioHead_\dimID)\le0\}.   
\end{equation}
Here, $f_j:\realPlus^\dimID\rightarrow\real$ for each $j\in\{1,\,\cdots,\,b\}$ is a function that defines the ball. 
Furthermore, we consider the worst-case expectation of the cost function within the ball (\ref{eq:someBall}).
\begin{equation}\label{eq:someWorst}
\innerMax. 
\end{equation}
Furthermore, we consider the Lagrange dual problem \cite[Section 5.2]{boyd2004convex} of the worst-case expectation (\ref{eq:someWorst}) for each $\xVar\in\setXVar$: 
\begin{equation}\label{eq:someDual}
\inf_{(\ineqVar_1,\,\cdots,\,\ineqVar_b,\,\eqVar)\in\realPlus^\dimID\times\real}\quad \lagrangeDualHead(\ineqVar_1,\,\cdots,\,\ineqVar_b,\,\eqVar), 
\end{equation}
\begin{equation}\label{eq:someLagrange}
\begin{split}
&\lagrangeDualHead(\ineqVar_1,\,\cdots,\,\ineqVar_b,\,\eqVar) \\
&=\sup_{(\densRatioHead_1,\,\cdots,\,\densRatioHead_\dimID)\in\realPlus^\dimID} 
\expectation{\refProb}{\densRatio\targetFcn} \\ 
&\quad\quad\quad-\sum_{j\in\{1,\,\cdots,\,b\}}\ineqVar_j f_j(\densRatioHead_1,\,\cdots,\,\densRatioHead_\dimID)+\eqVar(1-\expectation{\refProb}{\densRatio}). 
\end{split}
\end{equation}
Here, $\lagrangeDualHead:\setXVar\times\realPlus\times\realPlus\times\cdots\times\real\rightarrow\real\cup\{\infty\}$ is the Lagrange dual function \cite[Section 5.1]{boyd2004convex} associated with the worst-case expectation (\ref{eq:someWorst}). The symbol $\ineqVar_j\in\realPlus$ is the Lagrange multiplier that corresponds to the inequality constraint $f_j(\densRatioHead_1,\,\cdots,\,\densRatioHead_\dimID)\le0$ for each $j\in\{1,\,\cdots,\,b\}$. $\eqVar$ is also the Lagrange multiplier that corresponds to the equality constraint $1-\expectation{\refProb}{\densRatio}=0$. 
$\lagrangeDualHead(\ineqVar_1,\,\cdots,\,\ineqVar_b,\,\eqVar)$ in (\ref{eq:someLagrange}) is immediately  derived from the Lagrangian \cite[Section 5.2]{boyd2004convex}. 
\begin{lem}[Duality of the Worst Expectation in Some Balls]\label{lem:dual}
The Lagrange dual problem in (\ref{eq:someDual}) satisfies the following properties:
\begin{enumerate}
\renewcommand{\labelenumi}{(\roman{enumi})}
\item 
Suppose that the ball defined in (\ref{eq:someBall}) is a convex set and contains a strictly feasible point $\probHead\in\relint(\ambiguity)$ such that $f_j(\densRatioHead_1,\,\cdots,\,\densRatioHead_\dimID)<0$ for all $j\in\{1,\,\cdots,\,b\}$. 
Then, for each $\xVar\in\setXVar$, the worst-case expectation (\ref{eq:someWorst}) within the ball defined in (\ref{eq:someBall}) is equivalent to its dual problem in (\ref{eq:someDual}). 
\item If the cost function $\targetFcn$ is convex on $\setXVar$ for each $\ID\in\setID$, $\lagrangeDualHead(\ineqVar_1,\,\cdots,\,\ineqVar_b,\,\eqVar)$ is convex on $\setXVar\times\realPlus\times\realPlus\times\cdots\times\real$.
\item 
Suppose that $\sum_{j\in\{1,\,\cdots,\,b\}}\ineqVar_j f_j(\densRatioHead_1,\,\cdots,\,\densRatioHead_\dimID)$ is of class $C^1$ and a strictly convex function on $\realPlusPlus^\dimID$ for all $\ineqVar_j\ne0$ and for all $j\in\{1,\,\cdots,\,b\}$. For each $(\xVar,\,\ineqVar_1,\,\cdots,\,\ineqVar_b,\,\eqVar)\in\setXVar\times\realPlusPlus\times\realPlusPlus\times\cdots\real$,     
if there exists $(\densRatioHead_1^*,\,\cdots,\,\densRatioHead_\dimID^*)\in\realPlusPlus^\dimID$ where the gradient of the objective function in the right-hand side of (\ref{eq:someLagrange}) is zero, $(\densRatioHead_1^*,\,\cdots,\,\densRatioHead_\dimID^*)$ is the unique maximizer. 
\end{enumerate}
\end{lem}
}%textcolor

\begin{proof}[Proof of Lemma \ref{lem:dual}]

\textcolor{black}{
Let us prove the statement (i). 
The objective function $\inner$ is linear  in $\prob$ for each $\ID\in\setID$, and $\ambiguity$ is a convex set. Hence, $\innerMaxSL$ is a convex programming problem. Furthermore, $\relint(\ambiguity)$ contains a strictly feasible point. According to Slater's conditions\cite[Section 5.2.3]{boyd2004convex}, a strong duality emerges and the statement (i) holds. 
}%textcolor

\textcolor{black}{
Moreover, we prove the statement (ii) using the results from \cite[Section 3.2.3]{boyd2004convex}. This indicates that the pointwise supremum of the convex function is also convex. 
Therefore, the Lagrange dual function $\lagrangeDualHead(\ineqVar_1,\,\cdots,\,\ineqVar_b,\,\eqVar)$ is convex because the objective function of the right-hand side of (\ref{eq:someLagrange}) is convex for each $(\densRatioHead_1,\,\cdots,\,\densRatioHead_\dimID)$. 
Hence, the statement (ii) is proven. 
}%textcolor

\textcolor{black}{
Subsequently, we prove the statement (iii). From the assumption introduced in the statement (iii), the objective function in the right-hand side of (\ref{eq:someLagrange}) is clearly a strictly concave function of class $C^1$ in $(\densRatioHead_1,\,\cdots,\,\densRatioHead_\dimID)\in\realPlusPlus^\dimID$ for all $\ineqVar_j\ne0$, $j\in\{1,2,\}$. Therefore, at most one maximizer exists for the objective function \cite[Section 4.2.1]{boyd2004convex}. The existence of this maximizer follows from the assumption that there exists a point where the gradient of the objective function in the right-hand side of (\ref{eq:someLagrange}) is zero on $\realPlusPlus^\dimID$ for each $(\xVar,\,\ineqVar_1,\,\cdots,\,\ineqVar_b,\,\eqVar)\in\setXVar\times\realPlusPlus\times\realPlusPlus\times\cdots\real$. 
Hence, the statement (iii) is proven. 
}%textcolor
\end{proof}

\textcolor{black}{
\begin{rem}[Differentiable Balls]
Lemma \ref{lem:dual} (i) states that the worst-case expectation in (\ref{eq:someWorst}) can be reformulated as a minimization problem in (\ref{eq:someDual}). 
Furthermore, Lemma \ref{lem:dual} (ii) states that this reformulation can result in convex optimization. 
By Lemma \ref{lem:dual} (iii), solving (\ref{eq:someWorst}) reduces to identifying the stationary point that corresponds to the maximizer.  
\end{rem}
\begin{rem}[Differentiable Subsets of Total Variation Ball]
We can demonstrate that the weighted L2 ball $\ambiguity=\ambiguityL$ in (\ref{eq:ambiguityL}) and the DR ball $\ambiguity=\ambiguityDR$ in (\ref{eq:ambiguity}) are expressed by differentiable distance metrics that satisfy the sufficient conditions in Lemma \ref{lem:dual} (iii), in contrast to the TV ball $\ambiguityTV$.
\begin{equation}\label{eq:ambiguityTV}
\ambiguityTV = \{\probHead\in\probability\mid\expectation{\refProb}{|\densRatio-1|}\le\const\}. 
\end{equation}
In several cases, the TV ball is defined as half of the L1 ball. 
Both the weighted L2 and DR ball establish an inclusion relationship, as shown in Proposition \ref{prop:inclusion} in the Appendix.  
\end{rem}
%\begin{rem}[Proof Ideas of Lemma \ref{lem:dual}]
%The proof is based on the result of Slater's condition as shown in \cite{boyd2004convex}.
%\end{rem}
%\begin{rem}[Proof Idea of Lemma \ref{lem:dual}]
%\end{rem}
}%textcolor

\textcolor{black}{
\begin{lem}[Strong Duality of the Worst Expectation in L2 Balls]\label{lem:lagrangeDualL}
The following properties hold.
\begin{enumerate}
\renewcommand{\labelenumi}{(\roman{enumi})}
\item For every $\xVar\in\setXVar$, the function $\innerDualL$ in (\ref{eq:zeroWorstL}) holds:
\begin{multline}\label{eq:innerL}
\max_{\probHead\in\ambiguityL}\quad\inner = \\
\inf_{\KKTVar\in\realPlus\times\real}\quad \innerDualL.  
\end{multline}
%provided that the ball $\ambiguity$ is the weighted L2 ball defined in (\ref{eq:ambiguityL}). 
\item If the cost function $\targetFcn$ is convex on $\setXVar$ for each $\ID\in\setID$, the objective function in the right-hand side of (\ref{eq:innerL}), $\innerDualL$ is convex on $\setXVar\times\realPlus\times\real$.
\item Given $\xVar\in\setXVar$, suppose that $\min_{\ID\in\setID}\targetFcn>\expectation{\refProb}{\targetFcn}-\sqrt{\variance{\refProb}{\targetFcn}}/~\const$ is satisfied. %and $\targetFcn>\expectation{\refProb}{\targetFcn}-\sqrt{\variance{\refProb}{\targetFcn}}/~\const$ for all $\ID\in\setID$ are satisfied. 
The right-hand side of (\ref{eq:innerL}) is as follows:
\begin{multline}
\inf_{\KKTVar\in\realPlus\times\real}\quad \innerDualL \\
\begin{aligned}
=\min\Bigg\{&\max_{\ID\in\setID} \targetFcn,~ \\ 
&\expectation{\refProb}{\targetFcn}+\const\sqrt{\variance{\refProb}{\targetFcn}}\Bigg\}. 
\end{aligned}
\end{multline}
\item Given $\xVar\in\setXVar$, suppose that $\max_{\ID\in\setID}\targetFcn>\expectation{\refProb}{\targetFcn}+\const\sqrt{\variance{\refProb}{\targetFcn}}$ and $\min_{\ID\in\setID}\targetFcn>\expectation{\refProb}{\targetFcn}-\sqrt{\variance{\refProb}{\targetFcn}}/~\const$ are satisfied. Then, there exists some minimizer to the right-hand side of (\ref{eq:innerL}) that satisfies $\KKTVar\in\realPlusPlus\times\real$. 
 %the following pair is a solution of of $\KKTVar$ to the right-hand side of (\ref{eq:innerL}): 
%\begin{multline}\label{innerOptL}
%(\sqrt{\variance{\refProb}{\targetFcn}}/\,(2\const),\,\expectation{\refProb}{\targetFcn}) \\ \in\argmin_{\KKTVar\in\realPlus\times\real}\innerDualL.
%\end{multline}
\end{enumerate}
\end{lem}
}%textcolor

\begin{proof}[Proof of Lemma \ref{lem:lagrangeDualL}]
\textcolor{black}{
$\ambiguity=\ambiguityL$ in (\ref{eq:ambiguityL}) is a convex set. Furthermore, $\relint(\ambiguityL)$ is a non-empty set if $\const>0$. Hence, $\probHead\in\relint(\ambiguityL)$ exists that is strictly feasible; in particular, it satisfies $\expectation{\refProb}{(\densRatio-1)^2}<\const$ for all $\ID\in\setID$. By Lemma \ref{lem:dual} (i) and (ii), a strong dual problem arises as follows:
\begin{equation*}
\inf_{\KKTVar\in\realPlus\times\real}\lagrangeDualL=\max_{\probHead\in\ambiguityL}\quad\inner, 
\end{equation*}
\begin{equation*}%\label{eq:lagrangeDual}
%\lagrangeDual=\displaystyle\sup_{(\densRatioHead_1,\densRatioHead_2,\cdots\densRatioHead_\dimID)\in\realPlusPlus^\dimID}\quad\lagrangian,
\lagrangeDualL=\displaystyle\sup_{(\densRatioHead_1,\,\cdots,\,\densRatioHead_\dimID)\in\realPlus^\dimID}\quad\lagrangianL,
\end{equation*}
\begin{multline*}%\label{eq:lagrangian}
\lagrangianL
=\expectation{\refProb}{\densRatio\targetFcn} \\
-\ineqVar(\expectation{\refProb}{(1-\densRatio)^2}-\const^2)
+\eqVar(1-\expectation{\refProb}{\densRatio}).
\end{multline*}
Here, $\lagrangeDualHeadL:\setXVar\times\real\times\real \rightarrow \real\cup\{\infty\}$ and $\lagrangianHeadL: \setXVar\times\realPlus^\dimID\times\real\times\real \rightarrow \real$ are the Lagrange dual function and Lagrangian, respectively, associated with the problem in (\ref{eq:innerL}) with $\ambiguityL$. 
}%textcolor

\textcolor{black}{
The statement (i) can be proved by explicitly deriving the Lagrange dual function $\lagrangeDualL$. First, we consider the case $\ineqVar>0$. The Lagrangian is concave and quadratic in $\densRatio$; thus, the gradient of that in $\densRatio$ must be zero at the maximizer $\densRatio^*$, or $\densRatio^*$ must lie on the boundary of the domain $[0,\infty)$, as follows:
\begin{equation*}
\refProb\,(\targetFcn - 2\ineqVar(\densRatio^*{-1})-\eqVar)
\begin{cases}
=0, &(\densRatio^*>0), \\
\le0, &(\densRatio^*=0).
\end{cases}
\end{equation*}
Therefore, 
\begin{equation}\label{eq:worstDensL}
\begin{split}
\densRatio^*&=
\begin{cases}
\frac{\targetFcn+2\ineqVar-\eqVar}{2\ineqVar}, &(\targetFcn+2\ineqVar-\eqVar>0), \\
0, &(\targetFcn+2\ineqVar-\eqVar\le0),
\end{cases}
\\
&=\maxFcn{0,~\frac{\targetFcn+2\ineqVar-\eqVar}{2\ineqVar}}.
\end{split} 
\end{equation}
Hence, the Lagrangian becomes $\dualL$ in (\ref{eq:worstL}):
\begin{equation*}
\begin{split}
&\forall\ineqVar\in\realPlusPlus, \quad \\ &\lagrangianHeadL(\densRatioHead_1^*,\,\cdots,\,\densRatioHead_\dimID^*,\,\ineqVar,\,\eqVar) \\
&= 
\begin{cases}
\begin{aligned}
&\ineqVar\expectation{\refProb}{\left(\frac{\targetFcn+2\ineqVar-\eqVar}{2\ineqVar}\right)^2}\\
&+\ineqVar(\const^2-1)+\eqVar, 
\end{aligned}
&
\begin{aligned}
(\targetFcn&+2\ineqVar\\&-\eqVar>0), 
\end{aligned}
\\
\ineqVar(\const^2-1)+\eqVar, &
\begin{aligned}
(\targetFcn&+2\ineqVar\\&-\eqVar\le0), 
\end{aligned}
\end{cases}
\\
&= \ineqVar\expectation{\refProb}{\maxFcn{0,~\frac{\targetFcn+2\ineqVar-\eqVar}{2\ineqVar}}^2}+\ineqVar(\const^2-1)+\eqVar, \\
&= \dualL.
\end{split}
\end{equation*}
Second, let us consider the case $\ineqVar=0$. Then, the Lagrangian $\lagrangianL$ is affine in $\densRatioHead_\ID$ over $\realPlus$, and its value lies in the following intervals: 
\begin{equation*}
\begin{matrix}
\refProb\densRatio\targetFcn+\\
\quad\quad\quad
\refProb\eqVar(1-\densRatio)\in
\end{matrix}
\begin{cases}
[\refProb\eqVar,\,\infty), & \text{($\eqVar<\targetFcn$)}, \\
\{\refProb\eqVar\}, & \text{($\eqVar=\targetFcn$)}, \\
(-\infty,\,\refProb\eqVar], & \text{($\eqVar>\targetFcn$)}. \\
\end{cases}
\end{equation*}
Finally, by considering both cases, we can explicitly denote it as $\innerDualL$ in (\ref{eq:zeroWorstL}):
\begin{equation*}
\forall\ineqVar\in\realPlus, \quad \lagrangeDualL = \innerDualL.
\end{equation*}
Hence, the statement (i) is proven.
}%textcolor

\textcolor{black}{
Furthermore, we prove the statement (ii). This statement follows directly from Lemma \ref{lem:dual} (ii). 
}%textcolor

\textcolor{black}{
We also prove the statement (iii). The proof is based on the results of the KKT optimality conditions shown in \cite{boyd2004convex}. 
Let $(\ineqVar^*,\,\eqVar^*)$ denote a minimizer with respect to $\KKTVar$. 
}%textcolor

\textcolor{black}{
First, we consider the case that $\ineqVar^*=0$. Based on Lemma \ref{lem:lagrangeDualL} (i), the infimum value of $\innerDualL$ is equivalent to the maximum value of $\inner$ for each $\xVar$. By the definition of the target system, $\inner<\infty$; hence, it follows that $\dualExtendedHeadL(\xVar,\,\ineqVar^*,\,\eqVar^*)$ must also be finite. Therefore, we must have: 
\begin{equation*}
\forall \ID\in\setID, \quad \targetFcn\le\eqVar^*, \quad \dualExtendedHeadL(\xVar,\,0,\,\eqVar^*)=\eqVar^*,
\end{equation*}
that implies:
\begin{equation*}
\dualExtendedHeadL(\xVar,\,0,\,\eqVar^*) = \max_{\ID\in\setID} \quad \targetFcn.
\end{equation*}
}%textcolor

\textcolor{black}{
Second, we consider the other case that $\ineqVar^*>0$. According to the KKT optimality conditions \cite{boyd2004convex}, the following equations hold.
\begin{equation*}
\ineqVar^* (\const^2-\expectation{\refProb}{(1-\densRatio^*)^2})=0,
\end{equation*}
\begin{equation*}
1-\expectation{\refProb}{\densRatio^*}=0.
\end{equation*}
From $\ineqVar^*>0$ and $ \densRatio^*$ in (\ref{eq:worstDensL}), 
the KKT optimality conditions are reformulated as follows: 
\begin{equation*}
\expectation{\refProb}{\maxFcn{-1,~\frac{\targetFcn-\eqVar^*}{2\ineqVar^*}}^2}=\const^2, 
\end{equation*}
\begin{equation*}
\expectation{\refProb}{\maxFcn{-1,~\frac{\targetFcn-\eqVar^*}{2\ineqVar^*}}}=0.
\end{equation*}
Furthermore, from the assumption introduced in the statement (iii), the following inequality holds.  
\begin{equation}\label{eq:Assumption}
\min_{\ID\in\setID}\targetFcn>\expectation{\refProb}{\targetFcn}-\sqrt{\variance{\refProb}{\targetFcn}}/~\const.  
\end{equation}
Under this assumption, we observe that $\variance{\refProb}{\targetFcn}\ne0$. Additionally, we observe that the pair $(\ineqVar^*,\,\eqVar^*)$ is as follows:
\begin{equation}\label{eq:OPTineq}
\ineqVar^*=\frac{\sqrt{\variance{\refProb}{\targetFcn}}}{2\const}>0,
\end{equation}
\begin{equation}\label{eq:OPTeq}
\eqVar^*=\expectation{\refProb}{\targetFcn}, 
\end{equation}
because $(\targetFcn-\eqVar^*)/\,(2\ineqVar^*)>-1$ follows for all $\ID\in\setID$ by substituting (\ref{eq:OPTineq}) and (\ref{eq:OPTeq}) into (\ref{eq:Assumption}), and the KKT optimality conditions are reformulated as follows:
\begin{equation}\label{eq:KKTineq}
\expectation{\refProb}{\left(\frac{\targetFcn-\eqVar^*}{2\ineqVar^*}\right)^2}=\const^2, 
\end{equation}
\begin{equation}\label{eq:KKTeq}
\expectation{\refProb}{\frac{\targetFcn-\eqVar^*}{2\ineqVar^*}}=0.
\end{equation}
Hence, KKT optimality conditions (\ref{eq:KKTineq}) and (\ref{eq:KKTeq}), and these minimizers $(\ineqVar^*,\,\eqVar^*)$ yield the infimum value provided by:
\begin{equation*}
\begin{split}
\dualExtendedHeadL&\left(\xVar,\,\ineqVar^*,\,\eqVar^*\right) \\
&=\dualHeadL\left(\xVar,\,\ineqVar^*,\,\eqVar^*\right), \\ 
&=\ineqVar^*\expectation{\refProb}{\left(\maxFcn{-1,~\frac{\targetFcn-\eqVar^*}{2\ineqVar^*}}+1\right)^2} \\
&\quad +\ineqVar^*(\const^2-1)+\eqVar^*, \\
&=\ineqVar^*(\const^2+1)+\ineqVar^*(\const^2-1)+\eqVar^*, \\
&=\expectation{\refProb}{\targetFcn}+\const\sqrt{\variance{\refProb}{\targetFcn}}.
\end{split}
\end{equation*}
}%textcolor

\textcolor{black}{
Finally, by considering both cases, we can explicitly express it as follows:
\begin{multline*}
\innerMinL \\ 
\begin{aligned}
=\min\Bigg\{&\max_{\ID\in\setID} \targetFcn,~ \\
&\expectation{\refProb}{\targetFcn}+\const\sqrt{\variance{\refProb}{\targetFcn}}\Bigg\}.
\end{aligned}
\end{multline*}
Hence, the statement (iii) is proven. 
}%textcolor

\textcolor{black}{
Furthermore, we prove the statement (iv). We first introduce the assumption in the statement (iv) that $\max_{\ID\in\setID}\targetFcn>\expectation{\refProb}{\targetFcn}+\const\sqrt{\variance{\refProb}{\targetFcn}}$. Then, the minimum value is denoted by $(\ineqVar^*,\,\eqVar^*)\in\realPlusPlus\times\real$ in (\ref{eq:OPTineq}) and (\ref{eq:OPTeq}), hence:
\begin{equation*}
\begin{split}
&\innerMinL \\ &\quad\quad\quad=\expectation{\refProb}{\targetFcn}+\const\sqrt{\variance{\refProb}{\targetFcn}}, \\ %&\quad=\dualHeadL\Big(\xVar,\,\sqrt{\variance{\refProb}{\targetFcn}}/\,(2\const),\,\expectation{\refProb}{\targetFcn}\Big), \\ 
&\quad\quad\quad=\dualHeadL(\xVar,\,\ineqVar^*,\,\eqVar^*).
\end{split}
\end{equation*}
Hence, the statement (iv) is proven. 
}%textcolor
\end{proof}

\begin{proof}[Proof of Theorem \ref{thm:strongDualL}]
\textcolor{black}{
We first prove the statement (i). For each $\xVar$ in $\setXVar$, Lemma \ref{lem:lagrangeDualL} (i) states that the maximum value of (\ref{eq:DROC}) associated with the weighted L2 ball (\ref{eq:ambiguityL}) is equal to the infimum of (\ref{eq:dualDROCL}).  %Hence, the minimization regarding $\xVar$ in (\ref{eq:DROC}) associated with (\ref{eq:ambiguityL}), and (\ref{eq:dualDROCL}) are identical, yielding the first statement. 
This establishes the statement (i).
}%textcolor

\textcolor{black}{
Subsequently, we prove the statement (ii). This follows directly from the statements (i) and Lemma \ref{lem:lagrangeDualL} (iv). 
}%textcolor

\textcolor{black}{
Third, the statement (iii) is proved. By generalizing the results in \cite[Section 3.2.3]{boyd2004convex} to strictly convex functions, the objective function in (\ref{eq:DROC}) bocomes strictly convex if $\targetFcn$ is strictly convex. Therefore, the set of minimizers to the problem contains at most one point\cite[Section 4.2.1]{boyd2004convex}. Furthermore, the extreme value theorem \cite{keisler2012elementary} guarantees that the set of minimizers contains at least one point if $\targetFcn$ is continuous and $\setXVar$ is a bounded closed convex set. Assuming that the conditions for $\targetFcn$ and $\setXVar$ are satisfied, the minimizer must be unique.   
}%textcolor

\textcolor{black}{
Finally, we prove the statement (iv). We first introduce the assumption in the statement (iv) that $\targetFcn$ is convex and of  class $C^k$ on an open convex set $\setXVar$. Subsequently, Lemma \ref{lem:lagrangeDualL} (ii) states 
the objective function $\innerDualL$ is convex on $\setXVar\times\realPlus\times\real$. In addition, the function $\dualL$ is of class $C^1$. This is because, 
the derivative of the expectation in $\dualL$ is the sum of finitely many functions. 
%from the result of \cite[Proposition 2.1]{shapiro1994nondifferentiability}, the derivative and expectation can be interchanged for convex functions such as $\maxFcn{0,~(\targetFcn+2\ineqVar-\eqVar)/\,(2\ineqVar)}^2$ in $\dualL$; namely, the derivative of the expectation of any convex function is the expectation of the derivative of any convex function. 
Thus, the first derivative exists because the function $\dualL$ is continuous on all point, in particular at those satisfying $\targetFcn+2\ineqVar-\eqVar=0$. Hence, the statement (iv) is proven.
}%textcolor
\end{proof}

\begin{proof}[Proof of Theorem \ref{thm:variance}]
\textcolor{black}{
From Lemma \ref{lem:lagrangeDualL} (i), we substitute the worst-case expectation, $\max_{\probHead\in\ambiguityL}\inner$, associated with the weighted L2 ball (\ref{eq:ambiguityL}) into the infimum of $\dualExtendedL$. The statement follows directly by Lemma \ref{lem:lagrangeDualL} (iii) and the assumption that satisfies $\max_{\ID\in\setID}\targetFcnHead(\xVar^*,\,\ID)>\expectation{\refProb}{\targetFcnHead(\xVar^*,\,\ID)}+\const\sqrt{\variance{\refProb}{\targetFcnHead(\xVar^*,\,\ID)}}$ for the minimizer $\xVar^*$ to (\ref{eq:DROC}). 
%It is guaranteed that $\xVar^*$ is a Pareto-optimal solution to (\ref{eq:multi}) \cite{pardalos2017non}[Section 2.1]. 
%is not a  pareto solution, there exists some $\xVar\in\setXVar$ which satisfies either (\ref{eq:largeExpect}) and (\ref{eq:largeVariance}). It means that the following equation is satisfied:
%\begin{multline*}
%\expectation{\refProb}{\targetFcn}+  \const\sqrt{\variance{\refProb}{\targetFcn}}\\<\expectation{\refProb}{\targetFcnHead(\xVar^*,\,\ID)}+ \const\sqrt{\variance{\refProb}{\targetFcnHead(\xVar^*,\,\ID)}}.
%\end{multline*}
%However, the last equation contradicts the fact that $\xVar^*$ is the minimizer of (\ref{eq:variance}). Therefore, $\xVar^*$ is some pareto optimal  solution of (\ref{eq:multi}) Hence, the statement (iii) is proven. 
}%textcolor
\end{proof}

\subsection{Proofs of Theorems \ref{thm:strongDual}, \ref{thm:CVaR} and \ref{thm:knapsack}}
\label{sec:subResult2}

\textcolor{black}{
We prove Theorems \ref{thm:strongDual},  \ref{thm:CVaR} and \ref{thm:knapsack} by establishing Lemmas \ref{lem:ambiguity} and \ref{lem:lagrangeDual}.
We define the following set $\hatAmbiguity$ denoted as 
\begin{equation}\label{eq:hatAmbiguity}
\hatAmbiguity \coloneqq \{\probHead\in\probability\mid\forall \ID\in\setID, ~ \densRatio\lnFcn{\densRatio}\le\densRatio\lnFcn{1+\const}\}, 
\end{equation}
%\begin{rem}[The Constraint Can Be Defined in Zero]
where this study defines $0\lnFcn{0}=0$ because of the continuity as $\lim_{\densRatio\to0^+}\densRatio\lnFcn{\densRatio}=0$ \cite[Section 2.1]{cover1999elements}. 
\begin{lem}[Density-Ratio Balls Reformulation]\label{lem:ambiguity}
The set $\hatAmbiguity$ in (\ref{eq:hatAmbiguity}) is equivalent to the DR ball $\ambiguityDR$ in (\ref{eq:ambiguity}). 
\end{lem}
\begin{proof}[Proof of Lemma \ref{lem:ambiguity}]
If $\densRatio=0$, the statement immediately follows. If $\densRatio\ne0$, this statement follows directly from the strictly monotonicity of $\ln$. 
%Clearly, $\densRatio\lnFcn(\densRatio)\le\densRatio\lnFcn(1+\const)$ if $\densRatio=0$; then, we assume $\densRatio>0$. We have $\lnFcn(\densRatio)\le\lnFcn(1+\const)\Leftrightarrow \densRatio\le1+\const$ because $\ln$ is monotonically increasing. In addition, by multiplying $\densRatio>0$ to the left side, the statement is proven. 
\end{proof}
\begin{rem}[Strictly Convex Reformulations of Density-Ratio Balls]\label{rem:strictlyConvex}
The set $\hatAmbiguity$ in (\ref{eq:hatAmbiguity}) can be defined by strictly convex functions as $\densRatio\ln(\densRatio/\,(1+\const))\le0$ for all $\ID\in\setID$. This satisfies the properties of the existence of the maximizer to the Lagrange dual problem, as denoted in Lemma \ref{lem:dual} (iii). This strictly convex reformulation is based on the idea of a previous study \cite{hu2013kullback}. 
\end{rem}
\begin{lem}[Strong Duality of the Worst Expectation in DR Balls]\label{lem:lagrangeDual}
The following properties hold.
\begin{enumerate}
\renewcommand{\labelenumi}{(\roman{enumi})}
\item For every $\xVar\in\setXVar$, the function $\dualExtended$ in (\ref{eq:zeroWorst}) holds:
\begin{multline}\label{eq:inner}
\max_{\probHead\in\ambiguityDR}\quad\inner = \\
\inf_{\KKTVars\in\realPlus^\dimID\times\real}\innerDual. 
\end{multline}
%provided that the ball $\ambiguity$ is the DR ball defined in (\ref{eq:ambiguity}). 
\item If the cost function $\targetFcn$ is convex on $\setXVar$ for each $\ID\in\setID$, the objective function in the right-hand side of (\ref{eq:inner}) is convex on $\setXVar\times\realPlus^\dimID\times\real$.
\item The right-hand side of (\ref{eq:inner}) is denoted as follows:
\begin{multline}\label{eq:innerCVaR}
\innerMin \\=\inf_{\eqVar\in\real}\quad \innerCVaR,
\end{multline}
and there exists a minimizer of $\eqVar$ to the right-hand side of (\ref{eq:innerCVaR}) as follows:
\begin{multline}
\targetVaR \\ \in\argmin_{\eqVar\in\real}\innerCVaR,
\end{multline}
provided that the probability level is $\beta=\const/\,(1+\const)$.
\item Suppose that the probability level is $\beta=\const/\,(1+\const)$, that the reference distribution is uniform distribution $\refProb=1/\,\dimID$, and that $\constC=(1-\beta)\dimID$ is a positive integer. Then,
\begin{multline}
\max_{\probHead\in\ambiguityDR}\quad\inner \\ =
\max_{(\ID_1,\,\cdots,\,\ID_{\constC})\in\ambiguityC}\quad 
\sum_{l=1}^{\constC}\frac{\targetFcnHead(\xVar,\,\ID_l)}{\constC}. 
\end{multline}
%provided that the ball $\ambiguity$ is the DR ball defined in (\ref{eq:ambiguity}) and the set $\ambiguityC$ is defined in (\ref{eq:worstC}).
\end{enumerate}
\end{lem}
}%textcolor

\begin{proof}[Proof of Lemma \ref{lem:lagrangeDual}]
\textcolor{black}{
From Lemma \ref{lem:ambiguity}, we obtain the following equation.
\begin{equation*}
\innerHat=\max_{\probHead\in\ambiguityDR}\quad\inner. 
\end{equation*}
Furthermore, from Remark \ref{rem:strictlyConvex}, $\hatAmbiguity$ is a convex set. Additionally, if $\const>0$, $\relint(\hatAmbiguity)$ is a non-empty set. Therefore, there exists a $\probHead\in\relint(\hatAmbiguity)$ that is strictly feasible; that is, it satisfies $\densRatio<1+\const$ and $\densRatio\lnFcn{\densRatio}<\densRatio\lnFcn{1+\const}$ for all $\ID\in\setID$. From Lemmas \ref{lem:dual} (i) and (ii), a strong dual problem can be formulated as follows: 
\begin{equation*}
\inf_{\KKTVars\in\realPlus^\dimID\times\real}\lagrangeDual=\innerHat,
\end{equation*}
\begin{equation*}%\label{eq:lagrangeDual}
%\lagrangeDual=\displaystyle\sup_{(\densRatioHead_1,\densRatioHead_2,\cdots\densRatioHead_\dimID)\in\realPlusPlus^\dimID}\quad\lagrangian,
\lagrangeDual=\displaystyle\sup_{(\densRatioHead_1,\,\cdots,\,\densRatioHead_\dimID)\in\realPlus^\dimID}\quad\lagrangian,
\end{equation*}
\begin{multline*}%\label{eq:lagrangian}
\lagrangian
=\expectation{\refProb}{\densRatio\targetFcn} \\
-\expectation{\refProb}{\densRatio[\ineqVars]_\ID\lnFcn{\frac{\densRatio}{1+\const}}}
+\eqVar(1-\expectation{\refProb}{\densRatio}).
\end{multline*}
$\lagrangianHead: \setXVar\times\realPlus^\dimID\times\real^\dimID\times\real \rightarrow \real$, is the Lagrangian associated with the problem in (\ref{eq:inner}). 
}%textcolor

\textcolor{black}{
The statement (i) can be proved by considering two cases: $\ID\in\{l\in\setID\mid[\ineqVars]_l>0\}$ and $\ID\in\{l\in\setID\mid[\ineqVars]_l=0\}$ for each $\ID$. 
First, we consider the first case $\ID\in\{l\in\setID\mid[\ineqVars]_l>0\}$. Suppose that $\densRatio^*$ is the maximizer of $\densRatio$. The Lagrangian is concave for $\densRatio$. Then, the gradient in $\densRatio$ must be zero at $\densRatio^*$ as follows:
\begin{equation*}%\label{eq:gradientLagrangian}
%\forall \ID\in\{l\in\setID|[\ineqVars]_l>0\}, \quad 
%\derivativeScaler{\lagrangian}{\densRatio}|_{\densRatio=\densRatio^*} \\ =
\refProb\,\targetFcn 
-\refProb\,[\ineqVars]_\ID\left\{\lnFcn{\frac{\densRatio^*}{1+\const}}
-1\right\}-\refProb\eqVar
=0.
\end{equation*}
That being said, because this gradient involves $-\lnFcn{\densRatio^*/\,(1+\const)}$, there is no probability that $\densRatio^*$ lies on the boundary of the domain $\realPlus$. 
Therefore, 
\begin{equation}\label{eq:worstDens}
\densRatio^*=(1+\const)\expFcn{\frac{\targetFcn-[\ineqVars]_\ID-\eqVar}{[\ineqVars]_\ID}}. 
\end{equation}
Hence, the Lagrangian includes the following terms:
\begin{equation*}
\begin{split}
\refProb&\densRatio^*\targetFcn
-
\refProb\densRatio^*[\ineqVars]_\ID\lnFcn{\frac{\densRatio^*}{1+\const}} 
+\refProb\eqVar(1-\densRatio^*) \\
&=\refProb\densRatio^*\left(\targetFcn-\eqVar-[\ineqVars]_\ID\lnFcn{\frac{\densRatio^*}{1+\const}}\right) +\refProb\eqVar\\
&=\refProb(\densRatio^*[\ineqVars]_\ID +\eqVar)\\
&=\refProb\dual.
\end{split}
\end{equation*}
Subsequently, we consider the other case $\ID\in\{l\in\setID\mid[\ineqVars]_l=0\}$. The Lagrangian $\lagrangian$ is affine in $\densRatioHead_\ID$ on $\realPlus$ and its supremum is attained either at $\infty$ or $\refProb\eqVar$ as follows:
\begin{equation*}
\begin{matrix}
\refProb\densRatio\targetFcn+\\
\quad\quad\quad
\refProb\eqVar(1-\densRatio)\in
\end{matrix}
\begin{cases}
[\refProb\eqVar,\,\infty), & \text{($\eqVar<\targetFcn$)}, \\
\{\refProb\eqVar\}, & \text{($\eqVar=\targetFcn$)}, \\
(-\infty,\,\refProb\eqVar], & \text{($\eqVar>\targetFcn$)}. \\
\end{cases}
\end{equation*}
Finally, we can explicitly express both the cases as follows:
\begin{multline*}
\forall\ID\in\setID, \quad \\
 \begin{rcases}
 \refProb\eqVar, & ([\ineqVars]_\ID\ne0,~\targetFcn\ge\eqVar), \\
 \refProb\dual, & ([\ineqVars]_\ID>0), \\
 \infty, & ([\ineqVars]_\ID>0,~\targetFcn<\eqVar), 
 \end{rcases}\\
=\refProb\dualExtended.
\end{multline*}
Therefore, 
\begin{equation*}
\forall\ineqVars\in\realPlus^\dimID, \quad \lagrangeDual = \innerDual.
\end{equation*}
Hence, the statement (i) is proven. 
}%textcolor

\textcolor{black}{
We also prove the statement (ii) . This statement follows directly from Lemma \ref{lem:dual} (ii). 
}%textcolor

\textcolor{black}{
Furthermore, we prove the statement (iii). Suppose that $[\ineqVars^*]_\ID$ is the minimizer with respect to $[\ineqVars]_\ID$ for each $\ID$ in $\setID$.
}%textcolor

\textcolor{black}{
First, we consider the case that $\ID\in\{\ID\in\setID\mid[\ineqVars^*]_\ID=0\}$. Based on this lemma (i), the infimum of $\innerDual$ is equivalent to the maximum of $\inner$ for each $\xVar$. From the definition of the target system, $\inner<\infty$, $\dualExtendedHead(\xVar,\,\ID,\,[\ineqVars^*]_\ID,\,\eqVar)$ must be finite. Hence, 
\begin{multline*}
\forall \ID\in\{\ID\in\setID\mid[\ineqVars^*]_\ID=0\},\quad \\ \targetFcn\le\eqVar, \quad \dualExtendedHead(\xVar,\,\ID,\,0,\,\eqVar)=\eqVar.
\end{multline*}
}%textcolor

\textcolor{black}{
Second, we consider the case that $\ID\in\{\ID\in\setID\mid[\ineqVars^*]_\ID>0\}$. Based on the complementary slackness condition \cite{boyd2004convex}, the following equation holds:
\begin{equation*}
[\ineqVars^*]_\ID \, \densRatio^* \lnFcn{\frac{\densRatio^*}{1+\const}}=0.
\end{equation*}
From $\ID\in\{\ID\in\setID\mid[\ineqVars^*]_\ID>0\}$, it follows that:
\begin{equation*}
\densRatio^*=1+\const.
\end{equation*}
Therefore, by substituting $\densRatio^*$ in (\ref{eq:worstDens}) into the last equation, the following equation holds: 
\begin{equation*}
\forall \ID\in\{\ID\in\setID\mid[\ineqVars^*]_\ID>0\}, \quad [\ineqVars^*]_\ID = \targetFcn - \eqVar > 0.
\end{equation*}
Furthermore, by substituting the last equation into $\dualExtended$:
\begin{multline*}
\forall \ID\in\{\ID\in\setID\mid[\ineqVars^*]_\ID>0\}, \quad \targetFcn>\eqVar \quad \\
\dualExtendedHead(\xVar,\ID,[\ineqVars^*]_\ID,\eqVar)=(1+\const)\,(\targetFcn-\eqVar)+\eqVar.
\end{multline*}
}%textcolor

\textcolor{black}{
Finally, by considering both cases, we can explicitly express the infimum of $\dualExtended$ as follows:
\begin{equation*}
\begin{split}
\forall\ID&\in\setID, \quad  \inf_{[\ineqVars]_\ID\in\realPlus} \quad \dualExtended \\
 &=
 \begin{cases}
 \eqVar, &(\targetFcn\le\eqVar), \\
 (1+\const)(\targetFcn-\eqVar)+\eqVar, &(\targetFcn>\eqVar), 
 \end{cases}\\
 &= (1+\const)\maxFcn{0,~\targetFcn-\eqVar}+\eqVar.
\end{split}
\end{equation*}
We then introduce the assumption in the statement (iii) that $1+\const=(1-\beta)^{-1}$ holds. Moreover, we define $\innerCVaRFcnHead:\setXVar\times\real\rightarrow\real$ as follows:
\begin{equation}\label{eq:CVaRFcn}
\begin{split}
\innerCVaRFcn & \coloneqq \innerBetaCVaR, \\
& = \innerCVaR.
\end{split}
\end{equation}
Then, from the last equation, (\ref{eq:inner}) is equivalent to the following equation:
\begin{multline*}
\innerMin \\
=\inf_{\eqVar\in\real}\quad \innerCVaRFcn.
\end{multline*}
}%textcolor

\textcolor{black}{
$\innerCVaRFcn$ in (\ref{eq:CVaRFcn}) is a weighted sum of $\eqVar$ and $(1-\beta)^{-1}\maxFcn{0,~\targetFcn-\eqVar}$; these are convex in $\eqVar$. Therefore, $\innerCVaRFcn$ is also a convex function. Hence, from the result of \cite[Theorem 25.3]{rockafellar1997convex}, $\innerCVaRFcn$ is differentiable with respect to $\eqVar$ for all but perhaps countably many points over $\real$. Except at the countable set of points $\mathcal{D}\coloneqq\{\eqVar\mid\exists \ID\in\setID,~\targetFcn=\eqVar\}$, from the result of \cite[Lemma in Appendix]{rockafellar2000optimization}, the gradient of $\innerCVaRFcn$ with respect to $\eqVar$ can be computed as follows:
\begin{equation*}
\begin{split}
\forall \eqVar\notin\mathcal{D}, \quad  &\lim_{\delta\rightarrow0} \quad  \frac{\innerCVaRFcnHead(\xVar,\,\eqVar+\delta)-\innerCVaRFcn}{\delta} \\
&=1-(1-\beta)^{-1}\distribution{\refProb}{\targetFcn\ge\eqVar}, \\ 
&= (1-\beta)^{-1}(\distribution{\refProb}{\targetFcn<\eqVar}-\beta).
\end{split}
\end{equation*} 
From the result of \cite[Proposition 2.1]{shapiro1994nondifferentiability}, the right and left derivative of $\innerCVaRFcn$ over $\mathcal{D}$ can be calculated as follows:
\begin{equation*}
\begin{split}
\forall \eqVar\in\mathcal{D}, \quad &\lim_{\delta\rightarrow0^+} \quad  \frac{\innerCVaRFcnHead(\xVar,\,\eqVar+\delta)-\innerCVaRFcn}{\delta}, \\
&= (1-\beta)^{-1}(\distribution{\refProb}{\targetFcn<\eqVar}-\beta) \\
&\propto \distribution{\refProb}{\targetFcn<\eqVar}-\beta, 
\end{split}
\end{equation*}
\begin{equation*}
\begin{split}
\forall \eqVar\in\mathcal{D}, \quad &\lim_{\delta\rightarrow0^-} \quad  \frac{\innerCVaRFcnHead(\xVar,\,\eqVar+\delta)-\innerCVaRFcn}{\delta} \\
&= (1-\beta)^{-1}(\distribution{\refProb}{\targetFcn\le\eqVar}-\beta), \\
&\propto \distribution{\refProb}{\targetFcn\le\eqVar}-\beta.
\end{split}
\end{equation*}
}%textcolor

\textcolor{black}{
The minimizer of $\eqVar$ to $\innerCVaRFcn$ cannot be found by the first-order optimality condition because $\innerCVaRFcn$ is not differentiable. However, the existence of an extreme value of the convex function $\innerCVaRFcn$ can be established in the following. Assuming that there is no extremum value, the following inequalities must hold.
\begin{multline*}
\nexists \eqVar\in\mathcal{D}, \quad \\ \distribution{\refProb}{\targetFcn<\eqVar}-\beta\ge0, \quad  \distribution{\refProb}{\targetFcn\le\eqVar}-\beta\le0. \\ %\therefore \exists \eqVar\in\real, \quad  \distribution[\targetFcn<\eqVar]\le\beta\le\distribution[\targetFcn\le\eqVar].
\end{multline*}
However, from the definition of $\beta$-VaR, the following inequalities are satisfied.
\begin{equation*}
\distribution{\refProb}{\targetFcn\le\targetVaR}\ge\beta,
\end{equation*}
\begin{equation*}
\distribution{\refProb}{\targetFcn<\targetVaR}\le\beta.
\end{equation*}
%Subsequently, because of the defination of the target system, $\targetFcn<\infty$,  $0\le\distribution[\targetFcn<\eqVar]\le1$ must be satisfies and $\distribution[\targetFcn<\eqVar]$ monotonically increases with respect to $\eqVar$. $0\le\beta\le1$ is aslo satisfied. 
Hence, by replacing $\eqVar$ with $\targetVaR$, the extreme value must exist and contain $\targetVaR$ as follows:
\begin{equation*}
\targetVaR \in \argmin_{\eqVar\in\real} \innerCVaRFcn.
\end{equation*}
Hence, the statement (iii) is proven.
}%textcolor

\textcolor{black}{
We also prove the statement (iv). The statements (i) and (iii) state that for every $\xVar\in\setXVar$, the worst-case expectation is equivalent to the following problem. 
\begin{multline*}
\max_{\probHead\in\ambiguityDR}\quad\inner 
= (1-\beta)^{-1}\times  \\  
\expectation{\refProb}{\maxFcn{0,~\targetFcn - \targetVaR}} \\
+\targetVaR.
\end{multline*}
We introduce the assumption in the statement (iv) that $\beta=\const\,/(1+\const)$. Let us consider a collection $(\ID_1^*,\,\ID_2^*,\,\cdots,\,\ID_\constC^*)\in\ambiguityC$ that satisfies $\targetFcnHead(\xVar,\,\ID_1^*)\ge\targetFcnHead(\xVar,\,\ID_2^*)\ge\cdots\ge\targetFcnHead(\xVar,\,\ID_\constC^*)\ge\targetFcn$ for all $\ID\in\setID\setminus\{\ID_1^*,\,\ID_2^*,\,\cdots,\,\ID_\constC^*\}$. Then, because $\constC=\dimID\,(1-\beta)$ and the assumption $\refProb=1/\,\dimID$, $\distribution{\refProb}{\targetFcn\ge\targetFcnHead(\xVar,\,\ID_c^*)}=\constC/\,\dimID\ge1-\beta$ and $\distribution{\refProb}{\targetFcn<\targetFcnHead(\xVar,\,\ID_c^*)}\le\beta$ hold.
Therefore, from the definition of $\beta$-VaR , $\targetVaR=\targetFcnHead(\xVar,\,\ID_c^*)$ holds. Hence, the final problem is reformulated as follows:
\begin{equation*}
\begin{split}
&\innerMax \\
&=\frac{1}{\dimID\,(1-\beta)}\,(\targetFcnHead(\xVar,\,\ID_1^*)+\targetFcnHead(\xVar,\,\ID_2^*)+\cdots+\targetFcnHead(\xVar,\,\ID_c^*)), \\
&=\frac{1}{\constC}\,(\targetFcnHead(\xVar,\,\ID_1^*)+\targetFcnHead(\xVar,\,\ID_2^*)+\cdots+\targetFcnHead(\xVar,\,\ID_c^*)).
\end{split}
\end{equation*}
Since $(\ID_1^*,\,\ID_1^*,\,\cdots,\,\ID_c^*)\in\ambiguityC$ and $\targetFcnHead(\xVar,\,\ID_l^*)\ge\targetFcn$ for all $l\in\{1,\,2,\,\cdots,\,c\}$ and for all $\ID\in\setID\setminus(\ID_1^*,\,\ID_1^*,\,\cdots,\,\ID_c^*)$, the statement (iv) is proven.
}%textcolor
\end{proof}

\textcolor{black}{
\begin{rem}[Proof Ideas of Lemma \ref{lem:lagrangeDual}]
The proof is based on the results of the complementary slackness condition shown in \cite{boyd2004convex}. 
The difference between Lemma \ref{lem:lagrangeDual} (iii) and the results of previous studies \cite[Theorem 1] {rockafellar2000optimization},\cite[Proposition 5.11]{kuhn2024distributionallyrobustoptimization} is the differentiability of $\innerCVaRFcn$ in (\ref{eq:CVaRFcn}). 
The function $\innerCVaRFcn$ is not necessarily differentiable in $\eqVar\in\real$ because there exists some $\alpha\in\real$ that satisfies $\sum_{\ID\in\{\ID\in\setID\mid\targetFcn=\alpha\}}\refProb\ne0$, provided the distribution is discrete. 
\end{rem}
}%textcolor

\begin{proof}[Proof of Theorem \ref{thm:strongDual}]

\textcolor{black}{ 
We first prove the statement (i). Lemma \ref{lem:lagrangeDual} (i) states that the maximum value of (\ref{eq:DROC}) associated with (\ref{eq:ambiguity}) is equal to the infimum value of (\ref{eq:dualDROC}) for each $\xVar$ in $\setXVar$. This yields the first statement. %as follows:
}%textcolor

\textcolor{black}{
Subsequently, we prove the statement (ii). By naturally extending the results in \cite[Section 3.2.3]{boyd2004convex} to strictly convex functions, the objective function of (\ref{eq:DROC}) is strictly convex on $\setXVar$ if $\targetFcn$ is strictly convex. Therefore, the set of minimizers to the problem contains at most one point \cite[Section 4.2.1]{boyd2004convex}. In addition, the extreme value theorem \cite{keisler2012elementary} guarantees that the set of minimizers contains at least one point if $\targetFcn$ is continuous and $\setXVar$ is bounded and closed convex. Hence, it must be unique.%From Lemma %Suppose that $\targetFcn$ is strictly convex and continuous, and $\setXVar$ is bounded and closed convex. By naturally extending the result in \cite[Section 3.2.3]{boyd2004convex} to strictly convex functions, the objective function of (\ref{eq:DROC}) is also strictly convex on $\setXVar$. Therefore, the set of minimizers for the problem contains one point at most \cite[Section 4.2.1]{boyd2004convex}. In addition, according to the extreme value theorem \cite{keisler2012elementary}, the minimizer set contains at least one point; therefore, it must be unique.   
}%textcolor

\textcolor{black}{
Finally, the statement (iii) is proved. Lemma \ref{lem:lagrangeDual} (ii) states that the objective function $\innerDual$ is convex on $\setXVar\times\realPlus^\dimID\times\real$ if $\targetFcn$ is convex. In addition, the function $\dual$ is clearly of class $C^k$ because $\targetFcn$ is of class $C^k$ on $\setXVar\times\realPlusPlus^\dimID\times\real$. Hence, the statement (iii) is proven.   
%Suppose that $\targetFcn$ is convex and class $C^k$ on an open convex set $\setXVar$. Then, based on Lemma \ref{lem:lagrangeDual}, the objective function $\innerDual$ is convex on $\setXVar\times\realPlus^\dimID\times\real$. In addition, the function $\dual$ is clearly class $C^k$ because $\targetFcn$ is class $C^k$ on $\setXVar\times\realPlusPlus^\dimID\times\real$. Hence, the third statement is proven.   
}%textcolor
\end{proof}

\begin{proof}[Proof of Theorem \ref{thm:CVaR}]
\textcolor{black}{
From Lemma \ref{lem:lagrangeDual} (i) and (iii), the statement (i) can be proved by explicitly showing that $\innerCVaRFcnHead(\xVar,\,\targetVaR)$ is equal to $\targetTildeCVaR$ that satisfies (\ref{eq:tildeCVaR}). To simplify the notation, we use the following notation:
\begin{equation*}
\alpha_\beta(\xVar) \coloneqq \targetVaR.
\end{equation*}
The following equation holds.
\begin{equation*}
\innerMax = \innerCVaRFcnHead(\xVar,\,\alpha_\beta(\xVar)). 
\end{equation*}
}%textcolor

\textcolor{black}{
We verify that $\innerCVaRFcnHead(\xVar,\,\alpha_\beta(\xVar))$ satisfies the property of $\targetTildeCVaR$, denoted in (\ref{eq:tildeCVaR}). $\innerCVaRFcnHead(\xVar,\,\alpha_\beta(\xVar))$ can be calculated as follows: 
\begin{multline*}
\innerCVaRFcnHead(\xVar,\,\alpha_\beta(\xVar)) \\
=(1-\beta)^{-1}\,\expectation{\refProb}{\maxFcn{0,~\targetFcn-\alpha_\beta(\xVar)}} \\
+\alpha_\beta(\xVar). 
\end{multline*}
From the definition of $\beta$-VaR, $\distribution{\refProb}{\targetFcn\le\alpha_\beta(\xVar)}\ge\beta$ and $\distribution{\refProb}{\targetFcn<\alpha_\beta(\xVar)}\le\beta$ are immediately satisfied. Therefore, $\distribution{\refProb}{\targetFcn>\alpha_\beta(\xVar)}\le1-\beta$ and $\distribution{\refProb}{\targetFcn\ge\alpha_\beta(\xVar)}\ge1-\beta$ are satisfied. 
Therefore, the first expectation on the right-hand side of the last equation satisfies the following equation: 
\begin{equation}\label{eq:CVaRInequality}
\begin{split}
&\expectation{\refProb}{\targetFcn-\alpha_\beta(\xVar)\mid\targetFcn>\alpha_\beta(\xVar)} \\
&\quad\ge(1-\beta)^{-1}\,\sum_{\ID\in\{\ID\in\setID\mid\targetFcn>\alpha_\beta(\xVar)\}} \refProb(\targetFcn-\alpha_\beta(\xVar)) \\
&\quad=(1-\beta)^{-1}\,\expectation{\refProb}{\maxFcn{0,~\targetFcn-\alpha_\beta(\xVar)}} \\
&\quad=(1-\beta)^{-1}\,\sum_{\ID\in\{\ID\in\setID\mid\targetFcn\ge\alpha_\beta(\xVar)\}} \refProb(\targetFcn-\alpha_\beta(\xVar)) \\
&\quad\ge\expectation{\refProb}{\targetFcn-\alpha_\beta(\xVar)\mid\targetFcn\ge\alpha_\beta(\xVar)}.  
\end{split}
\end{equation} 
Therefore:
\begin{equation*}
\begin{split}
&\expectation{\refProb}{\targetFcn-\alpha_\beta(\xVar)\mid\targetFcn>\alpha_\beta(\xVar)}+\alpha_\beta(\xVar) \\
&\quad\ge \innerCVaRFcnHead(\xVar,\,\alpha_\beta(\xVar)) \\
&\quad\ge\expectation{\refProb}{\targetFcn-\alpha_\beta(\xVar)\mid\targetFcn\ge\alpha_\beta(\xVar)}+\alpha_\beta(\xVar).  
\end{split}
\end{equation*}
Therefore: 
\begin{equation*}
\begin{split}
&\expectation{\refProb}{\targetFcn\mid\targetFcn>\alpha_\beta(\xVar)} \\
&\quad\ge \innerCVaRFcnHead(\xVar,\,\alpha_\beta(\xVar)) \\
&\quad\ge\expectation{\refProb}{\targetFcn~|~\targetFcn\ge\alpha_\beta(\xVar)}.   
\end{split}
\end{equation*}
Therefore: 
\begin{multline*}
\targetHatCVaR 
\ge \innerCVaRFcnHead(\xVar,\,\alpha_\beta(\xVar)) \\
\ge \targetCVaR.
\end{multline*}
Hence, $\innerCVaRFcnHead(\xVar,\,\alpha_\beta(\xVar))$ clearly satisfies the property of $\targetTildeCVaR$ and the statement (i) is proven. 
}%textcolor

\textcolor{black}{
We prove the statement (ii) as follows.  
If $\distribution{\refProb}{\targetFcn\le\alpha_\beta(\xVar)}=\beta$  is satisfied, the following equation holds in (\ref{eq:CVaRInequality}). 
\begin{multline*}
\expectation{\refProb}{\targetFcn-\alpha_\beta(\xVar)\mid\targetFcn>\alpha_\beta(\xVar)} \\
\quad=(1-\beta)^{-1}\,\expectation{\refProb}{\maxFcn{0,~\targetFcn-\alpha_\beta(\xVar)}}.
\end{multline*}
Therefore, by adding $\alpha_\beta(\xVar)$ to both sides in the last equation, the following equation holds.
\begin{equation*}
\targetHatCVaR = \innerCVaRFcnHead(\xVar,\,\alpha_\beta(\xVar)). 
\end{equation*}
Hence, the statement is proven.
}%textcolor
\end{proof}

\begin{proof}[Proof of Theorem \ref{thm:knapsack}]
\textcolor{black}{
Lemma \ref{lem:lagrangeDual} (iv) states that for every $\xVar\in\setXVar$, the DDRO problem in (\ref{eq:DROC}) associated with the DR ball (\ref{eq:ambiguity}) is equivalent to the deterministic RC problem in (\ref{eq:knapsackRC2}) associated with the worst $\constC$ costs.  
}%textcolor
\end{proof}

\section{Numerical Experiments}

\textcolor{black}{
This section presents numerical experiments on patroller-agent design problems, which have been extended from the worst-case mean hitting time minimization \cite{diaz2023distributed}, and formulated as DDRO problems. We compare the performance of the proposed DDRO method $(\const > 0)$ with that of the conventional SOC method.
}%textcolor

%\subsection{Description of Patrolling Tasks}
%We introduce patrolling tasks; in which, a robotic patroller agent is assigned to continuously visit a set of different location. To maximize its surveillance performance, metrics such as the first hitting time \cite{norris1998markov} are commonly used to measure the surveillance performance. Here, the first hitting time is defined, in the context of a Markov chain, as the time it takes to first reach a designated goal state. 

\subsection{Settings of Patroller-Agent Design}

\textcolor{black}{
We introduce patrolling tasks in which a robotic patroller agent is assigned to visit different locations continuously.
}%textcolor

\textcolor{black}{
We consider a finite undirected graph $\mathcal{G}(\setID,\,\mathcal{E})$ and a discrete-time Markov chain that models the transitions of a patroller agent state. The agent's state at time $t\in\{0,\,1,\,\cdots\}$ is denoted by  $X_t\in\setID$, where $\setID=\{1,\,2,\,\cdots,\,\dimID\}$ is the set of nodes (states), and $\mathcal{E}\subseteq\setID\times\setID$ is the set of edges (connections between the states).  
}%textcolor

\textcolor{black}{
The Markov chain is characterized by a transition matrix $\boldsymbol{P}\in\real^{\dimID\times\dimID}$, where the component $[\boldsymbol{P}]_{j,\,k}$  denotes the transition probability from state $j$ to state $k$: $[\boldsymbol{P}]_{j,\,k}=\distribution{}{X_t=k \mid X_{t-1}=j}$. 
The chain satisfies the Markov  property: namely,  $\distribution{}{X_t=j_t \mid X_{t-1}=j_t-1}=\distribution{}{X_t=j_t \mid X_{t-1}=j_{t-1},\, \cdots,\, X_{0}=j_{0}}$ for all $j_t\in\setID$.
}%textcolor

\textcolor{black}{
We introduce the mean hitting time minimization problem for a patrolling agent on a graph as studied in \cite{diaz2023distributed}. In a Markov chain, the mean hitting time is defined as the expected time for the patrolling agent to first reach a designated goal states $\mathcal{A}(\ID)\coloneq\{\ID\}\subseteq\setID$, as defined in \cite[Equation (3)]{diaz2023distributed}: 
\begin{equation}\label{eq:hittingTime}
\targetFcn=\boldsymbol{\pi}^\top(\eye_{\dimID} - \boldsymbol{E}_{\mathcal{A}}(\ID) \boldsymbol{P} \boldsymbol{E}_{\mathcal{A}}(\ID))\boldsymbol{\delta}_{\mathcal{A}}(\ID). 
\end{equation}
Here, $\xVar=\vect(\boldsymbol{P})$ is the decision variable, and $\boldsymbol{\delta}_{\mathcal{A}}(\ID)\in\real^\dimID$ is a vector valued in $\{0,1\}$. The component $[\boldsymbol{\delta}_{\mathcal{A}}(\ID)]_j=1$ if $j\notin\mathcal{A}(\ID)$; otherwise, $[\boldsymbol{\delta}_{\mathcal{A}}(\ID)]_j=0$. Futhermore, $\boldsymbol{E}_{\mathcal{A}}(\ID)=\diag(\boldsymbol{\delta}_{\mathcal{A}}(\ID))\in\real^{\dimID\times\dimID}$. 
}%textcolor

\textcolor{black}{
Let $\mathcal{M}_{\boldsymbol{\pi},\,\mathcal{E}}^*$ denote the set of irreducible and reversible stochastic matrices with a particular stationary distribution $\boldsymbol{\pi}\in\realPlus^\dimID$. Suppose that $\boldsymbol{P} \in \mathcal{M}_{\boldsymbol{\pi},\,\mathcal{E}}^*$. We restrict the connections in all matrices $\boldsymbol{P}\in\mathcal{M}_{\boldsymbol{\pi},\,\mathcal{E}}^*$ to $[P]_{j,\,k}=0$ for all $(j,\,k)\notin\mathcal{E}$. The component $[\boldsymbol{\pi}]_j$ represents the long-run proportion of time the patroller spends in state $j\in\setID$. From the results in \cite[Lemma 3.1]{diaz2023distributed}, the mean hitting time $\targetFcn$ is convex in $\xVar$ over $\mathcal{M}_{\boldsymbol{\pi},\,\mathcal{E}}^*$. 
}%textcolor

\textcolor{black}{
%As defined in \cite{diaz2023distributed}, the mean hitting time to a goal state $\mathcal{A}(\ID)=\{\ID\}\subseteq\setID$ is denoted as follows:
}%textcolor

\subsection{DDRO problems of Patroller-Agent Design}

\begin{figure}[t]
   \centering
   %\framebox{\parbox{3in}{We suggest that you use a text box to insert a graphic (which is ideally a 300 dpi TIFF or EPS file, with all fonts embedded) because, in an document, this method is somewhat more stable than directly inserting a picture.}}
   \includegraphics[scale=0.85,clip]{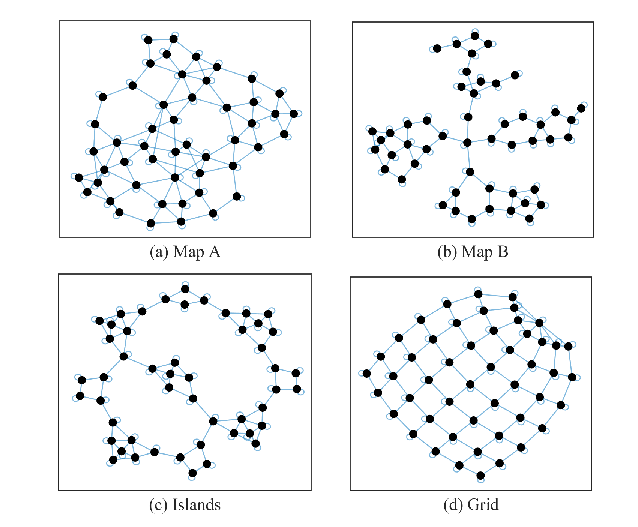}
   \caption{Graphs of the four topologies in \cite[Fig.\,3]{almeida2004recent}. In the figure, the black dots represent the nodes and the solid lines represent the edges.}
   \label{figurelabel}
\end{figure}

\textcolor{black}{
Consider the DDRO problem in (\ref{eq:DROC}), where the cost function is defined as the mean hitting time $\targetFcn$ in (\ref{eq:hittingTime}). Originally, the objective was to minimize the worst-case mean hitting time or the weighted sum of the mean hitting time \cite[Equation (5) and (6)]{diaz2023distributed}.
However, these weights for all goal states $\mathcal{A}(\ID)\subseteq\setID$ are often difficult to assign because noteworthy nodes are not known in prior. Instead of assigning these weights, we formulate the DDRO patroller-agent design problem by treating a probability distribution $\prob$ as the weights for all $\mathcal{A}(\ID)\subseteq\setID$: 
\begin{equation}\label{eq:hittingTimeDROC}
\begin{split}
&\min_{\boldsymbol{P}\in\mathcal{M}_{\boldsymbol{\pi},\,\mathcal{E}}^*} \quad \max_{\probHead\in\ambiguity} \quad \expectation{\prob}{\boldsymbol{\pi}^\top(\eye_{\dimID} - \boldsymbol{E}_{\mathcal{A}}(\ID) \boldsymbol{P} \boldsymbol{E}_{\mathcal{A}}(\ID) \boldsymbol{\delta}_{\mathcal{A}}(\ID)}. 
%&\text{s.t.}~[\boldsymbol{P}]_{j,\,k}=0,\quad \forall (j,\,k)\notin\mathcal{E}.
\end{split}
\end{equation}
We then consider two types of $\ambiguity$: the weighted L2 ball $\ambiguity=\ambiguityL$ defined in (\ref{eq:ambiguityL}) and the DR ball $\ambiguity=\ambiguityDR$ defined in (\ref{eq:ambiguity}). 
In both cases, the center of the ball $\refProb$ is set to a uniform distribution. 
If we choose $\ambiguity$ as the weighted L2 ball in (\ref{eq:ambiguityL}), the DDRO patroller-agent design problem is equivalent to the expectation- and standard-deviation-based problem described in Theorem {\ref{thm:variance}}: 
\begin{equation}\label{eq:hittingTimeVariance}
\begin{split}
\min_{\boldsymbol{P}\in\mathcal{M}_{\boldsymbol{\pi},\,\mathcal{E}}^*} 
&
 \expectation{\refProb}{\boldsymbol{\pi}^\top(\eye_{\dimID} - \boldsymbol{E}_{\mathcal{A}}(\ID) \boldsymbol{P} \boldsymbol{E}_{\mathcal{A}}(\ID) \boldsymbol{\delta}_{\mathcal{A}}(\ID)} \\ 
&+\const\sqrt{\variance{\refProb}{\boldsymbol{\pi}^\top(\eye_{\dimID} - \boldsymbol{E}_{\mathcal{A}}(\ID) \boldsymbol{P} \boldsymbol{E}_{\mathcal{A}}(\ID) \boldsymbol{\delta}_{\mathcal{A}}(\ID)}}. 
\end{split}
\end{equation}
If the size of ball $\const$ is small, the solutions to the last problem are Pareto-optimal solutions to the following multi-objective optimization problem, as described in Corollary {\ref{cor:variance}}: 
\begin{equation}
\begin{split}
\min_{\boldsymbol{P}\in\mathcal{M}_{\boldsymbol{\pi},\,\mathcal{E}}^*} 
\Bigg\{&
 \expectation{\refProb}{\boldsymbol{\pi}^\top(\eye_{\dimID} - \boldsymbol{E}_{\mathcal{A}}(\ID) \boldsymbol{P} \boldsymbol{E}_{\mathcal{A}}(\ID) \boldsymbol{\delta}_{\mathcal{A}}(\ID)}, ~ \\ 
&\sqrt{\variance{\refProb}{\boldsymbol{\pi}^\top(\eye_{\dimID} - \boldsymbol{E}_{\mathcal{A}}(\ID) \boldsymbol{P} \boldsymbol{E}_{\mathcal{A}}(\ID) \boldsymbol{\delta}_{\mathcal{A}}(\ID)}}
\Bigg\}. 
\end{split}
\end{equation}
}%textcolor

\textcolor{black}{
If $\ambiguity$ is the DR ball in (\ref{eq:ambiguity}), the DDRO problem becomes equivalent to the CVaR minimization problem described in Theorem \ref{thm:CVaR}:
\begin{equation}\label{eq:hittingTimeCVaR}
\begin{split}
&\min_{\boldsymbol{P}\in\mathcal{M}_{\boldsymbol{\pi},\,\mathcal{E}}^*} \quad \hatCVaR{\beta}{\refProb}{\boldsymbol{\pi}^\top(\eye_{\dimID} - \boldsymbol{E}_{\mathcal{A}}(\ID) \boldsymbol{P} \boldsymbol{E}_{\mathcal{A}}(\ID) \boldsymbol{\delta}_{\mathcal{A}}(\ID)}. 
%&\text{s.t.}~[\boldsymbol{P}]_{j,\,k}=0,\quad \forall (j,\,k)\notin\mathcal{E}.
\end{split}
\end{equation}
Here, $\beta\text{-}\hat{\text{CVaR}}$ represents the conditional expectation of the mean hitting time exceeding $\beta$-VaR, where  $\beta$-VaR denotes the mean hitting time to the worst nodes with probability greater than $\beta$. Theorems {\ref{thm:variance}} and {\ref{thm:CVaR}} state that the size of ball, $\const$, controls the weight parameter $\const$ in the standard deviation based formulation, and the probability level $\beta=\const/\,(1+\const)$ in the CVaR formulation. Furthermore, Theorems \ref{thm:strongDualL} and \ref{thm:strongDual} state that the DDRO problem can be reformulated as a single-layer smooth convex optimization problem, (\ref{eq:dualDROCL}) and (\ref{eq:dualDROC}), because $\targetFcn$ in (\ref{eq:hittingTime}) is convex and smooth. 
}%textcolor

\textcolor{black}{
Theorem \ref{thm:knapsack} provides an alternative interpretation of the DDRO problem in (\ref{eq:hittingTimeCVaR}) as a  deterministic RC problem involving multiple worst-case nodes. Specifically, this formulation assumes that the reference distribution $\refProb$ is uniform, as stated in Theorem \ref{thm:knapsack}.
%as follows: 
%\begin{equation}\label{eq:hittingTimeRC}
%\begin{split}
%&\min_{\boldsymbol{P}\in\mathcal{M}_{\boldsymbol{\pi}}^*} ~ \max_{(\ID_1,\,\cdots,\,\ID_\constC)\in\ambiguityC} ~ \boldsymbol{\pi}^\top(\eye_{\dimID} - \boldsymbol{E}_{\mathcal{A}}(\ID) \boldsymbol{P} \boldsymbol{E}_{\mathcal{A}}(\ID) \boldsymbol{\delta}_{\mathcal{A}}(\ID), \\
%&\text{s.t.}~[\boldsymbol{P}]_{j,\,k}=0,\quad \forall (j,\,k)\notin\mathcal{E},
%\end{split}
%\end{equation}
%\begin{equation}
%\ambiguityC\coloneqq\{(\ID_1,\,\cdots,\,\ID_\constC)\in\setID^\constC~|~\forall(j,\,k)\in\setID\times\setID,~\ID_j\ne\ID_k\}.
%\end{equation}
%Here, $\constC$ is a positive integer which satisfies $\constC=\dimID/\,(1+\const)$, and $\ambiguityC$ denotes a set which contains the worst $\constC$ nodes.
}%textcolor

\textcolor{black}{
We consider the graphs presented in  \cite[Fig.\,3]{almeida2004recent}, which include four different topologies with $|\setID|=\dimID=50$, as shown in Fig. \ref{figurelabel}. The stationary distribution $\boldsymbol{\pi}$ of the patroller agent's Markov chain is set to the uniform distribution.
}%textcolor

\subsection{Verification of Solvability and Interpretability}

\begin{figure}[t]
   \centering
   %\framebox{\parbox{3in}{We suggest that you use a text box to insert a graphic (which is ideally a 300 dpi TIFF or EPS file, with all fonts embedded) because, in an document, this method is somewhat more stable than directly inserting a picture.}}
   \includegraphics[scale=1.0,clip]{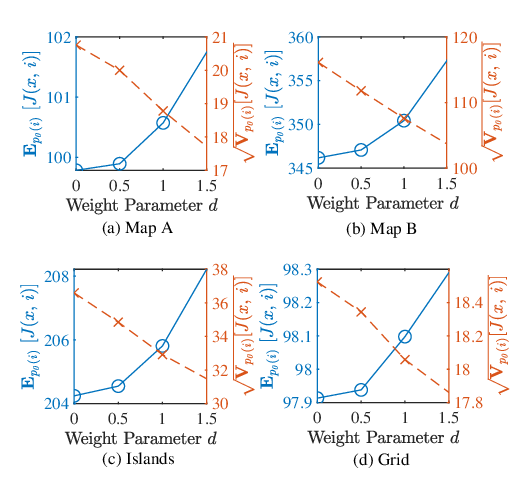}
   \caption{
Results of the expectation and the standard deviation of mean hitting time. The solid lines represents the expectation of mean hitting time $\expectation{\refProb}{\targetFcn}$, whereas the dashed lines represent the standard deviation $\sqrt{\variance{\refProb}{\targetFcn}}$. %The area to the left of the dotted line represents the region that satisfies the preconditions of Theorem \ref{thm:variance}.
   }
   \label{figure:result}
\end{figure}

\begin{figure}[t]
   \centering
   %\framebox{\parbox{3in}{We suggest that you use a text box to insert a graphic (which is ideally a 300 dpi TIFF or EPS file, with all fonts embedded) because, in an document, this method is somewhat more stable than directly inserting a picture.}}
   \includegraphics[scale=1,clip]{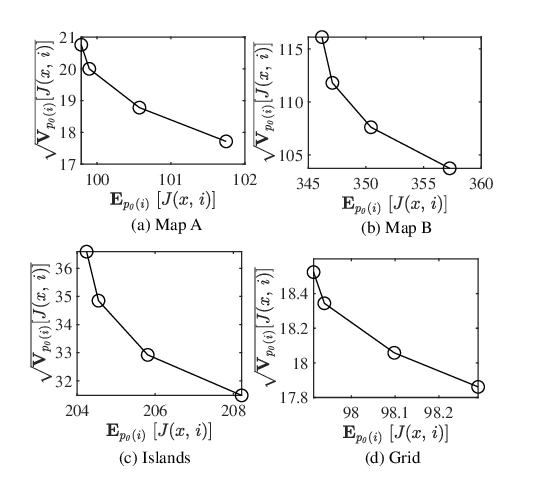}
   \caption{
Pareto front of the expectation and the standard deviation of mean hitting time. The circles represent Pareto-optimal solutions that the proposed method have found. The horizontal axis is the expectation of mean hitting time $\expectation{\refProb}{\targetFcn}$, while the vertical axis is the standard deviation $\sqrt{\variance{\refProb}{\targetFcn}}$.  %The area to the left of the dotted line represents the region that satisfies the preconditions of Theorem \ref{thm:variance}.
   }
   \label{figure:multi}
\end{figure}

\begin{table}[t]
\caption{Results of $\beta^{\prime}$-CVaR of mean hitting time corresponds to each target probability level $\beta^{\prime}$. Here, $\beta$ is a designed probability level used in the proposed method.}
\label{table:resultCVaR}
\begin{center}
\begin{tabular}{l||c|c|c||c}
\hline

\multicolumn{5}{l}{Map A ($|\setID|\,(=\dimID)=50,~|\mathcal{E}|=154$)}  \\
\hline
\multirow{2}{*}{Probability, $\beta^{\prime}$} & \multicolumn{3}{c||}{Proposed Method} & SOC \\ \cline{2-4}
    & $\beta=0.98$ & $\beta=0.75$ & $\beta=0.50$ & Method \\

\hline
$\beta^{\prime}=0.98$ & $\boldsymbol{140.0}$ & $154.1$ & $157.0$ & $159.5$ \\
\hline
$\beta^{\prime}=0.75$ & $129.5$ & $\boldsymbol{124.3}$ & $125.1$ & $126.5$ \\
\hline
$\beta^{\prime}=0.50$ & $120.8$ & $116.2$ & $\boldsymbol{115.8}$ & $116.9$ \\
\hline
$\beta^{\prime}=0$ & $104.3$ & $101.3$ & $100.4$ & $\boldsymbol{99.7}$ \\
\hline \hline

\multicolumn{5}{l}{Map B ($|\setID|\,(=\dimID)=50,~|\mathcal{E}|=118$)}  \\
\hline
\multirow{2}{*}{Probability, $\beta^{\prime}$} & \multicolumn{3}{c||}{Proposed Method} & SOC \\ \cline{2-4}
    & $\beta=0.98$ & $\beta=0.75$ & $\beta=0.50$ & Method \\

\hline
$\beta^{\prime}=0.98$ & $\boldsymbol{576.6}$ & $644.1$ & $656.5$ & $669.5$ \\
\hline
$\beta^{\prime}=0.75$ & $523.7$ & $\boldsymbol{479.2}$ & $489.0$ & $499.8$ \\
\hline
$\beta^{\prime}=0.50$ & $467.0$ & $435.2$ & $\boldsymbol{425.9}$ & $429.2$ \\
\hline
$\beta^{\prime}=0$ & $386.3$ & $363.5$ & $348.8$ & $\boldsymbol{346.1}$ \\
\hline \hline

\multicolumn{5}{l}{Islands ($|\setID|\,(=\dimID)=50,~|\mathcal{E}|=132$)}  \\
\hline
\multirow{2}{*}{Probability, $\beta^{\prime}$} & \multicolumn{3}{c||}{Proposed Method} & SOC \\ \cline{2-4}
    & $\beta=0.98$ & $\beta=0.75$ & $\beta=0.50$ & Method \\

\hline
$\beta^{\prime}=0.98$ & $\boldsymbol{252.8}$ & $259.8$ & $263.6$ & $265.7$ \\
\hline
$\beta^{\prime}=0.75$ & $245.9$ & $\boldsymbol{243.5}$ & $246.2$ & $250.2$ \\
\hline
$\beta^{\prime}=0.50$ & $237.4$ & $233.0$ & $\boldsymbol{231.6}$ & $233.7$ \\
\hline
$\beta^{\prime}=0$ & $211.4$ & $209.1$ & $205.6$ & $\boldsymbol{204.2}$ \\
\hline \hline

\multicolumn{5}{l}{Grid ($|\setID|\,(=\dimID)=50,~|\mathcal{E}|=141$)}  \\
\hline
\multirow{2}{*}{Probability, $\beta^{\prime}$} & \multicolumn{3}{c||}{Proposed Method} & SOC \\ \cline{2-4}
    & $\beta=0.98$ & $\beta=0.75$ & $\beta=0.50$ & Method \\

\hline
$\beta^{\prime}=0.98$ & $\boldsymbol{146.6}$ & $152.0$ & $153.2$ & $154.2$ \\
\hline
$\beta^{\prime}=0.75$ & $125.4$ & $\boldsymbol{120.8}$ & $121.1$ & $121.4$ \\
\hline
$\beta^{\prime}=0.50$ & $116.3$ & $112.0$ & $\boldsymbol{111.8}$ & $121.1$ \\
\hline
$\beta^{\prime}=0$ & $101.1$ & $98.1$ & $98.0$ & $\boldsymbol{97.9}$ \\
\hline
\end{tabular}
\end{center}
\end{table}

\textcolor{black}{
We employ the fmincon function in MATLAB \cite{coleman1999optimization} to solve the SOC and DDRO problems. As described in Theorems \ref{thm:strongDualL} and \ref{thm:strongDual}, such solvers are capable of obtaining globally optimal solutions for general smooth convex optimization problems, including those defined by the weighted L2 and DR balls in (\ref{eq:dualDROCL}) and (\ref{eq:dualDROC}), provided that the Lagrange multiplier is non-negative.
To find globally optimal solutions that satisfy the non-negativity of the Lagrange multiplier, we use the logarithmic barrier function, $-0.1\ln(\ineqVar)$, as described in \cite[Section 11.2.1]{boyd2004convex}.
}%textcolor

\textcolor{black}{
From the results in Fig. \ref{figure:result}, we can observe that the weight parameter $\const$ can balance the expectation and standard deviation of the mean hitting time.  In particular, when $\const$ is in the range of  $0\text{--}1.5$, 
Fig. \ref{figure:multi} shows that the proposed method effectively obtained the Pareto front characterized by up to a 3\% change in the expected mean hitting time and up to a 14\% change in its standard deviation. This provides supporting evidence for the interpretability of the weighted L2 ball as described in Theorem \ref{thm:variance}. 
 %As described in (\ref{eq:hittingTimeVariance}) and Theorem \ref{thm:variance} (i), Table \ref{table:resultVariance} shows that, the proposed method successfully obtains sub-optimal solutions that balance expectation and variance when $\const$ is small, and worst-case solutions when $\const$ is large, as described in Theorem \ref{thm:variance}, which addresses the interpretability of the weighted L2 ball. 
}%textcolor

\textcolor{black}{
From the results in Table \ref{table:resultCVaR}, we confirm that the proposed method effectively solves the $\beta$-CVaR minimization problem for each value of $\beta$, as described in Theorem \ref{thm:CVaR} (i) (interpretability of the DR ball). Notably, the method achieves significant improvement in the $98\%$-CVaR ($\beta^\prime=0.98$) of the mean hitting time on Map B, reducing it by  $69.9$ steps, which corresponds to a $13\%$ decrease when compared to the conventional SOC method.
}%textcolor

\section{Conclusions}
\textcolor{black}{
In this study, we propose DDRO problems associated with two types of uncertainty sets: weighted L2 balls and density-ratio balls. The sizes of these balls are determined by the trade-off parameter between the expected performance and variability of performance, and the probability level that provides the worst-case cost exceeding a certain threshold.  Furthermore, the proposed method is reduced to single-layer smooth convex programming problems with only the constraint of non-negativity of the Lagrange multiplier. The numerical experiments on the DDRO patroller-agent design problems, associated with the defined balls, demonstrated the practical applicability of the proposed method by identifying a Pareto front with respect to the mean and standard deviation of the mean hitting time, and achieving a reduction in CVaR. 
}%textcolor

\textcolor{black}{
This study focuses on DDRO problems without constraints related to distributional uncertainties. Problems involving distributionally robust constraints remain important topics for future studies. 
Another challenge is analyzing the regret bounds\cite{9992705} of a distributionally robust optimal controller. 
}%textcolor

\textcolor{black}{
Beyond the distributionally robust optimization setting addressed in this study, the proposed method has the potential to be extended to other stochastic control problems that consider both performance and variability, such as risk-sensitive controls. This approach is particularly applicable to complex numerical optimization tasks involving multi-objective formulations that balance the expected performance and its variability. 
}%textcolor

\addtolength{\textheight}{-12cm}   % This command serves to balance the column lengths

%%%%%%%%%%%%%%%%%%%%%%%%%%%%%%%%%%%%%%%%%%%%%%%%%%%%%%%%%%%%%%%%%%%%%%%%%%%%%%%%

%%%%%%%%%%%%%%%%%%%%%%%%%%%%%%%%%%%%%%%%%%%%%%%%%%%%%%%%%%%%%%%%%%%%%%%%%%%%%%%%

%%%%%%%%%%%%%%%%%%%%%%%%%%%%%%%%%%%%%%%%%%%%%%%%%%%%%%%%%%%%%%%%%%%%%%%%%%%%%%%%

\section*{APPENDIX}
\textcolor{black}{
\begin{prop}[Inclusion Relationship Between the Balls]\label{prop:inclusion}
The weighted L2 ball in (\ref{eq:ambiguityL}), DR ball in (\ref{eq:ambiguity}), and TV ball in (\ref{eq:ambiguityTV}) satisfy the following properties:
\begin{enumerate}
\renewcommand{\labelenumi}{(\roman{enumi})}
\item The weighted L2 ball (\ref{eq:ambiguityL}) is a subset of the TV ball (\ref{eq:ambiguityTV}).
\item The DR ball (\ref{eq:ambiguity}) is a subset of both the weighted L2 ball (\ref{eq:ambiguityL}) and TV ball (\ref{eq:ambiguityTV}) if $\const\ge1$.
\end{enumerate}
\end{prop}
}%textcolor

\begin{proof}[Proof of Proposition \ref{prop:inclusion}]
\textcolor{black}{
Using Jensen's inequality \cite[Section 3.1.8]{boyd2004convex}, and the concavity of the square root function, we obtain the following relationship between weighted L2 and TV distances:
\begin{equation}
\begin{split}
\sqrt{\expectation{\refProb}{(\densRatio-1)^2}} &  \ge \expectation{\refProb}{\sqrt{(\densRatio-1)^2}} \\ &=\expectation{\refProb}{|\densRatio-1|}.
\end{split} 
\end{equation}
Hence, the statement (i) is proven. 
}%textcolor

\textcolor{black}{ 
Moreover, any distribution within the DR ball (\ref{eq:ambiguity}) belongs to the weighted L2 ball 
if $\const\ge1$ because:
\begin{equation*}
\forall\ID\in\setID, \quad \densRatio\le1+\const \Rightarrow \expectation{\refProb}{\sqrt{(\densRatio-1)^2}}\le\const. 
\end{equation*}
This fact and the statement (i) imply that the DR ball $\ambiguityDR$ in (\ref{eq:ambiguity}) is a subset of both the weighted L2 ball $\ambiguityL$ in (\ref{eq:ambiguityL}) and the TV ball $\ambiguityTV$ in (\ref{eq:ambiguityTV}). Hence, the statement (ii) is proven. 
}%textcolor
\end{proof}

%\section*{ACKNOWLEDGMENT}

%The preferred spelling of the word �acknowledgment� in America is without an �e� after the �g�. Avoid the stilted expression, �One of us (R. B. G.) thanks . . .�  Instead, try �R. B. G. thanks�. Put sponsor acknowledgments in the unnumbered footnote on the first page.

%%%%%%%%%%%%%%%%%%%%%%%%%%%%%%%%%%%%%%%%%%%%%%%%%%%%%%%%%%%%%%%%%%%%%%%%%%%%%%%%

%References are important to the reader; therefore, each citation must be complete and correct. If at all possible, references should be commonly available publications.

\bibliography{DDROC2025}

\begin{thebibliography}{10}
\providecommand{\url}[1]{#1}
\csname url@rmstyle\endcsname
\providecommand{\newblock}{\relax}
\providecommand{\bibinfo}[2]{#2}
\providecommand\BIBentrySTDinterwordspacing{\spaceskip=0pt\relax}
\providecommand\BIBentryALTinterwordstretchfactor{4}
\providecommand\BIBentryALTinterwordspacing{\spaceskip=\fontdimen2\font plus
\BIBentryALTinterwordstretchfactor\fontdimen3\font minus
  \fontdimen4\font\relax}
\providecommand\BIBforeignlanguage[2]{{%
\expandafter\ifx\csname l@#1\endcsname\relax
\typeout{** WARNING: IEEEtran.bst: No hyphenation pattern has been}%
\typeout{** loaded for the language `#1'. Using the pattern for}%
\typeout{** the default language instead.}%
\else
\language=\csname l@#1\endcsname
\fi
#2}}

\bibitem{11107818}
Y.~Shida and Y.~Ito, ``Discrete distributionally robust optimal control with
  explicitly constrained optimization,'' in \emph{2025 American Control
  Conference (ACC)}, 2025, pp. 1418--1424.

\bibitem{duan2021markov}
X.~Duan and F.~Bullo, ``Markov chain--based stochastic strategies for robotic
  surveillance,'' \emph{Annual Review of Control, Robotics, and Autonomous
  Systems}, vol.~4, no.~1, pp. 243--264, 2021.

\bibitem{diaz2023distributed}
G.~D{\'\i}az-Garc{\'\i}a, F.~Bullo, and J.~R. Marden, ``Distributed markov
  chain-based strategies for multi-agent robotic surveillance,'' \emph{IEEE
  Control Systems Letters}, vol.~7, pp. 2527--2532, 2023.

\bibitem{bertsekas1996stochastic}
D.~Bertsekas and S.~E. Shreve, \emph{Stochastic optimal control: the
  discrete-time case}.\hskip 1em plus 0.5em minus 0.4em\relax Athena
  Scientific, 1996, vol.~5.

\bibitem{crespo2003stochastic}
L.~G. Crespo and J.-Q. Sun, ``Stochastic optimal control via bellman's
  principle,'' \emph{Automatica}, vol.~39, no.~12, pp. 2109--2114, 2003.

\bibitem{sastry2011adaptive}
S.~Sastry and M.~Bodson, \emph{Adaptive control: stability, convergence and
  robustness}.\hskip 1em plus 0.5em minus 0.4em\relax Courier Corporation,
  2011.

\bibitem{slotine1985robust}
J.-J.~E. Slotine, ``The robust control of robot manipulators,'' \emph{The
  International Journal of Robotics Research}, vol.~4, no.~2, pp. 49--64, 1985.

\bibitem{scherer2001theory}
C.~Scherer, ``Theory of robust control,'' \emph{Delft University of
  Technology}, pp. 1--160, 2001.

\bibitem{taskesen2024distributionally}
B.~Taskesen, D.~Iancu, {\c{C}}.~Ko{\c{c}}yi{\u{g}}it, and D.~Kuhn,
  ``Distributionally robust linear quadratic control,'' \emph{Advances in
  Neural Information Processing Systems}, vol.~36, 2024.

\bibitem{liu2023data}
R.~Liu, G.~Shi, and P.~Tokekar, ``Data-driven distributionally robust optimal
  control with state-dependent noise,'' in \emph{2023 IEEE/RSJ International
  Conference on Intelligent Robots and Systems (IROS)}.\hskip 1em plus 0.5em
  minus 0.4em\relax IEEE, 2023, pp. 9986--9991.

\bibitem{yang2020wasserstein}
I.~Yang, ``Wasserstein distributionally robust stochastic control: A
  data-driven approach,'' \emph{IEEE Transactions on Automatic Control},
  vol.~66, no.~8, pp. 3863--3870, 2020.

\bibitem{nishimura2021rat}
H.~Nishimura, N.~Mehr, A.~Gaidon, and M.~Schwager, ``Rat ilqr: A risk
  auto-tuning controller to optimally account for stochastic model mismatch,''
  \emph{IEEE Robotics and Automation Letters}, vol.~6, no.~2, pp. 763--770,
  2021.

\bibitem{nguyen2023distributionally}
H.~T. Nguyen and D.-H. Choi, ``Distributionally robust model predictive control
  for smart electric vehicle charging station with v2g/v2v capability,''
  \emph{IEEE Transactions on Smart Grid}, vol.~14, no.~6, pp. 4621--4633, 2023.

\bibitem{coulson2021distributionally}
J.~Coulson, J.~Lygeros, and F.~D{\"o}rfler, ``Distributionally robust chance
  constrained data-enabled predictive control,'' \emph{IEEE Transactions on
  Automatic Control}, vol.~67, no.~7, pp. 3289--3304, 2021.

\bibitem{hu2013kullback}
Z.~Hu and L.~J. Hong, ``Kullback-leibler divergence constrained
  distributionally robust optimization,'' \emph{Available at Optimization
  Online}, vol.~1, no.~2, p.~9, 2013.

\bibitem{pilipovsky2024distributionally}
J.~Pilipovsky and P.~Tsiotras, ``Distributionally robust density control with
  wasserstein ambiguity sets,'' \emph{arXiv preprint arXiv:2403.12378}, 2024.

\bibitem{gao2023distributionally}
R.~Gao and A.~Kleywegt, ``Distributionally robust stochastic optimization with
  wasserstein distance,'' \emph{Mathematics of Operations Research}, vol.~48,
  no.~2, pp. 603--655, 2023.

\bibitem{shafieezadeh2019regularization}
S.~Shafieezadeh-Abadeh, D.~Kuhn, and P.~M. Esfahani, ``Regularization via mass
  transportation,'' \emph{Journal of Machine Learning Research}, vol.~20, no.
  103, pp. 1--68, 2019.

\bibitem{cherukuri2022data}
A.~Cherukuri, A.~Zolanvari, G.~Banjac, and A.~R. Hota, ``Data-driven
  distributionally robust optimization over a network via distributed
  semi-infinite programming,'' in \emph{2022 IEEE 61st Conference on Decision
  and Control (CDC)}.\hskip 1em plus 0.5em minus 0.4em\relax IEEE, 2022, pp.
  4771--4775.

\bibitem{luo2017decomposition}
F.~Luo and S.~Mehrotra, ``Decomposition algorithm for distributionally robust
  optimization using wasserstein metric,'' \emph{arXiv preprint
  arXiv:1704.03920}, 2017.

\bibitem{mohajerin2018data}
P.~Mohajerin~Esfahani and D.~Kuhn, ``Data-driven distributionally robust
  optimization using the wasserstein metric: Performance guarantees and
  tractable reformulations,'' \emph{Mathematical Programming}, vol. 171, no.~1,
  pp. 115--166, 2018.

\bibitem{van2015distributionally}
B.~P. Van~Parys, D.~Kuhn, P.~J. Goulart, and M.~Morari, ``Distributionally
  robust control of constrained stochastic systems,'' \emph{IEEE Transactions
  on Automatic Control}, vol.~61, no.~2, pp. 430--442, 2015.

\bibitem{kishida2022risk}
M.~Kishida and A.~Cetinkaya, ``Risk-aware linear quadratic control using
  conditional value-at-risk,'' \emph{IEEE Transactions on Automatic Control},
  vol.~68, no.~1, pp. 416--423, 2022.

\bibitem{liu2019discrete}
Y.~Liu, A.~Pichler, and H.~Xu, ``Discrete approximation and quantification in
  distributionally robust optimization,'' \emph{Mathematics of Operations
  Research}, vol.~44, no.~1, pp. 19--37, 2019.

\bibitem{farokhi2023distributionally}
F.~Farokhi, ``Distributionally robust optimization with noisy data for discrete
  uncertainties using total variation distance,'' \emph{IEEE Control Systems
  Letters}, vol.~7, pp. 1494--1499, 2023.

\bibitem{miao2021data}
F.~Miao, S.~He, L.~Pepin, S.~Han, A.~Hendawi, M.~E. Khalefa, J.~A. Stankovic,
  and G.~Pappas, ``Data-driven distributionally robust optimization for vehicle
  balancing of mobility-on-demand systems,'' \emph{ACM Transactions on
  Cyber-Physical Systems}, vol.~5, no.~2, pp. 1--27, 2021.

\bibitem{wiesemann2014distributionally}
W.~Wiesemann, D.~Kuhn, and M.~Sim, ``Distributionally robust convex
  optimization,'' \emph{Operations research}, vol.~62, no.~6, pp. 1358--1376,
  2014.

\bibitem{reemtsen2013semi}
R.~Reemtsen and J.-J. R{\"u}ckmann, \emph{Semi-infinite programming}.\hskip 1em
  plus 0.5em minus 0.4em\relax Springer Science \& Business Media, 2013,
  vol.~25.

\bibitem{mehrotra2014cutting}
S.~Mehrotra and D.~Papp, ``A cutting surface algorithm for semi-infinite convex
  programming with an application to moment robust optimization,'' \emph{SIAM
  Journal on Optimization}, vol.~24, no.~4, pp. 1670--1697, 2014.

\bibitem{reemtsen1991discretization}
R.~Reemtsen, ``Discretization methods for the solution of semi-infinite
  programming problems,'' \emph{Journal of Optimization Theory and
  Applications}, vol.~71, pp. 85--103, 1991.

\bibitem{zhang2021efficient}
Z.~Zhang, S.~Ahmed, and G.~Lan, ``Efficient algorithms for distributionally
  robust stochastic optimization with discrete scenario support,'' \emph{SIAM
  Journal on Optimization}, vol.~31, no.~3, pp. 1690--1721, 2021.

\bibitem{villani2009optimal}
C.~Villani \emph{et~al.}, \emph{Optimal transport: old and new}.\hskip 1em plus
  0.5em minus 0.4em\relax Springer, 2009, vol. 338.

\bibitem{rockafellar2000optimization}
R.~T. Rockafellar, S.~Uryasev, \emph{et~al.}, ``Optimization of conditional
  value-at-risk,'' \emph{Journal of risk}, vol.~2, pp. 21--42, 2000.

\bibitem{jacobson2003optimal}
D.~Jacobson, ``Optimal stochastic linear systems with exponential performance
  criteria and their relation to deterministic differential games,'' \emph{IEEE
  Transactions on Automatic control}, vol.~18, no.~2, pp. 124--131, 2003.

\bibitem{whittle1981risk}
P.~Whittle, ``Risk-sensitive linear/quadratic/gaussian control,''
  \emph{Advances in Applied Probability}, vol.~13, no.~4, pp. 764--777, 1981.

\bibitem{boyd2004convex}
S.~Boyd and L.~Vandenberghe, \emph{Convex optimization}.\hskip 1em plus 0.5em
  minus 0.4em\relax Cambridge university press, 2004.

\bibitem{pardalos2017non}
P.~M. Pardalos, A.~{\v{Z}}ilinskas, J.~{\v{Z}}ilinskas, \emph{et~al.},
  \emph{Non-convex multi-objective optimization}.\hskip 1em plus 0.5em minus
  0.4em\relax Springer, 2017.

\bibitem{kuhn2024distributionallyrobustoptimization}
\BIBentryALTinterwordspacing
D.~Kuhn, S.~Shafiee, and W.~Wiesemann, ``Distributionally robust
  optimization,'' 2024. [Online]. Available:
  \url{https://arxiv.org/abs/2411.02549}
\BIBentrySTDinterwordspacing

\bibitem{keisler2012elementary}
H.~J. Keisler, \emph{Elementary calculus: An infinitesimal approach}.\hskip 1em
  plus 0.5em minus 0.4em\relax Courier Corporation, 2012.

\bibitem{cover1999elements}
T.~M. Cover, \emph{Elements of information theory}.\hskip 1em plus 0.5em minus
  0.4em\relax John Wiley \& Sons, 1999.

\bibitem{rockafellar1997convex}
R.~T. Rockafellar, \emph{Convex analysis}.\hskip 1em plus 0.5em minus
  0.4em\relax Princeton university press, 1997, vol.~28.

\bibitem{shapiro1994nondifferentiability}
A.~Shapiro and Y.~Wardi, ``Nondifferentiability of the steady-state function in
  discrete event dynamic systems,'' \emph{IEEE transactions on Automatic
  Control}, vol.~39, no.~8, pp. 1707--1711, 1994.

\bibitem{almeida2004recent}
A.~Almeida, G.~Ramalho, H.~Santana, P.~Tedesco, T.~Menezes, V.~Corruble, and
  Y.~Chevaleyre, ``Recent advances on multi-agent patrolling,'' in
  \emph{Advances in Artificial Intelligence--SBIA 2004: 17th Brazilian
  Symposium on Artificial Intelligence, Sao Luis, Maranhao, Brazil, September
  29-Ocotber 1, 2004. Proceedings 17}.\hskip 1em plus 0.5em minus 0.4em\relax
  Springer, 2004, pp. 474--483.

\bibitem{coleman1999optimization}
T.~Coleman, M.~A. Branch, and A.~Grace, ``Optimization toolbox,'' \emph{For Use
  with MATLAB. User’s Guide for MATLAB 5, Version 2, Relaese II}, 1999.

\bibitem{9992705}
A.~Karapetyan, A.~Iannelli, and J.~Lygeros, ``On the regret of h$\infty$
  control,'' in \emph{2022 IEEE 61st Conference on Decision and Control (CDC)},
  2022, pp. 6181--6186.

\end{thebibliography}
\bibliographystyle{IEEEtran}

\end{document}